\documentclass[preprint]{imsart}

\usepackage{amsmath,amsthm,amssymb,amsfonts,latexsym,hyperref}
\usepackage{stmaryrd}
\usepackage{graphicx}
\usepackage{natbib}
\bibpunct{(}{)}{,}{a}{}{;}
\usepackage[usenames]{color}
\usepackage{geometry}

\newcommand\eps{\varepsilon}
\renewcommand\phi{\varphi}
\newcommand\pro[1]{\mathbb{P}\left(#1\right)}
\newcommand\esp[1]{\mathbb{E}\left[#1\right]}

\newcommand\espcc[2]{\mathbb{E}_{#1}\left[#2\right]}

\newcommand\uno[1]{1_{\left\{#1\right\}}}

\newtheorem{thm}{Theorem}[section]
\newtheorem{prop}[thm]{Proposition}

\newtheorem{assu}{Assumption}
\newtheorem{lem}[thm]{Lemma}
\newtheorem{rem}[thm]{Remark}

\newcommand\n{\mathbb{N}}
\renewcommand\r{\mathbb{R}}

\newcommand\E{\mathcal{E}}

\renewcommand\ll{\left}
\newcommand\rr{\right}

\begin{document}
	
	\begin{frontmatter}
		
		\title{Annealed limit for a diffusive disordered mean-field model with random jumps}
		
		\runtitle{Annealed limit in random environment}
		
		\begin{aug}
			
			\author{\fnms{Xavier} \snm{Erny}\ead[label=e4]{xavier.erny@polytechnique.edu}}

			\address{Ecole polytechnique, Centre de math\'ematiques appliqu\'ees (CMAP), 91128 Palaiseau
			}
			
			\runauthor{X. Erny}
			
		\end{aug}

		\begin{abstract}
			We study a sequence of $N-$particle mean-field systems, each driven by $N$ simple point processes $Z^{N,i}$ in a random environment. Each $Z^{N,i}$ has the same intensity $(f(X^N_{t-}))_t$ and at every jump time of $Z^{N,i},$ the process $X^N$ does a jump of height $U_i/\sqrt{N}$ where the $U_i$ are disordered centered random variables attached to each particle. We prove the convergence in distribution of $X^N$ to some limit process $\bar X$ that is solution to an SDE with a random environment given by a Gaussian variable, with a convergence speed for the finite-dimensional distributions. This Gaussian variable is created by a CLT as the limit of the patial sums of the~$U_i.$ To prove this result, we use a coupling for the classical CLT relying on the result of \cite{komlos_approximation_1976}, that allows to compare the conditional distributions of $X^N$ and $\bar X$ given the environment variables, with the same Markovian technics as the ones used in \cite{erny_mean_2022}.
		\end{abstract}


		
	\end{frontmatter}


\section{Introduction}

The term disordered comes from physics litterature to designate some "asymmetric systems". Particle systems in random environment can often be seen as disordered models, since,
in this kind of models, it is possible to attach to each particle (or each pair of particles) a random variable, interpreted as a disordered variable, creating asymmetric interactions between the particles.
This is for example the case of models with spin glass dynamics, where for each pair of particle~$(i,j),$ there exists a disordered variable of the form $S_iS_jV_{ij}$, where $S_i$ and $S_j$ are typically $\{-1,+1\}-$valued variables modelling the spins of the particles~$i$ and~$j$, and $V_{ij}$ is a random variable modelling an interaction strength between the two particles. This kind of model has been introduced in the seminal paper of \cite{sherrington_solvable_1975}.

In the model we study in this paper, unlikely with spin glass dynamics model, we attach disordered variables to each particle (and not to each pair of particles) that are i.i.d. and independent of the time, like in the very similar model of~\cite{pfaffelhuber_mean-field_2022}. More precisely, we study the solutions of stochastic differential equations with a drift term and a jump term that depends on the disordered variables. There is a natural application of this kind of model in neurosciences (as explained in Section~4 of \cite{pfaffelhuber_mean-field_2022}), where the equations model the dynamics of the membrane potentials of neurons, the jump times model the times at which the neurons receive spikes by another neuron of the network, and the jump heights (given by the disordered variables) model the synaptic weights in the network. The drift term models the (deterministic) dynamics of the membrane potential between its jump times.

The property we want to prove is the convergence in distribution of the disordered system as the number of particles of the system goes to infinity. As we work in a random environment, the convergence of the processes can be understood in two ways: a quenched convergence (i.e. convergence conditionally on the environment) or an annealed convergence (i.e. convergence when the environment variables are averaged). We refer to~\cite{ben_arous_large_1995} and~\cite{guionnet_averaged_1997} for the definitions of the terms "quenched" and "annealed" that we use in this paper.

Formally, the aim of the paper is to prove the convergence in distribution of a process~$(X^N_t)_t$ defined in a random environment. Let us define rigorously this process. If the environment $(u^{[N]}_j)_{1\leq j\leq N}\in \r^N$ is fixed, we define $(X^N_t(u^{[N]}))_t$ as the solution of the following SDE:
\begin{equation}\label{XNu}
dX^N_t(u^{[N]}) = b(X^N_t(u^{[N]}))dt + \frac{1}{\sqrt{N}}\sum_{j=1}^N u^{[N]}_j\int_{\r_+} \uno{z\leq f(X^N_{t-}(u^{[N]}))}d\pi_j(t,z),
\end{equation}
where $b$ and $f$ are deterministic functions and $\pi_j$ ($1\leq j\leq N$) are independent Poisson measures on $[0,+\infty)^2$ of intensity~$dt\cdot dz,$ that are independent of~$X^N_0.$ We note $\nu^N_0$ the distribution of~$X^N_0.$ By definition, the process $(X^N_t)_t$ is defined as
\begin{equation}\label{XN}
X^N_t := X^N_t(U^{[N]}),
\end{equation}
where the variables $U^{[N]}_j$ ($1\leq j\leq N$) are i.i.d. centered random variables with finite variance~$\sigma^2$, that are independent of the Poisson measures~$\pi_j$ ($1\leq j\leq N$) and of the initial condition~$X^N_0$. We note $\mu$ the law of these variables. This law is independent of~$N$ and is a parameter of the model. For the sake of notation, instead of defining for each $N-$particle system a family of~$N$ variables $U^{[N]}_j$ ($1\leq j\leq N$), we introduce a countable sequence of random variables $U_j$ ($j\geq 1$). By definition, we set $U^{[N]} := (U_1,...,U_N)$ (e.g. the first $N$ variables of $U^{[N+1]}$ are exactly $U^{[N]}$). In particular the~$[N]$ in superscript will be dropped.


The limit of~$(X^N_t)_t$ is shown to be a process~$(\bar X_t)_t$ that is also defined in a random environment. For a fixed~$w\in\r,$ let $(\bar X_t(w))_t$ be defined as the solution of
\begin{equation}\label{barXw}
d\bar X_t(w) = b(\bar X_t(w))dt + wf(\bar X_t(w))dt + \sigma\sqrt{f(\bar X_t(w))}dB_t,
\end{equation}
where $B$ is a standard one-dimensional Brownian motion that is independent of~$\bar X_0.$ We note~$\bar\nu_0$ the distribution of~$\bar X_0.$ The limit process~$(\bar X_t)_t$ is defined as
\begin{equation}\label{barX}
\bar X_t := \bar X_t(W),
\end{equation}
where $W$ is a Gaussian variable with parameter~$(0,\sigma^2)$ that is independent of~$B,\bar X_0.$

Heuristically, a simple way to obtain the limit equation~\eqref{barXw}$-$\eqref{barX} from~\eqref{XNu}$-$\eqref{XN} consists in writing the SDE of~$(X^N_t)_t$ as
\begin{equation}\label{XNheuristique}
dX^N_t = b(X^N_t)dt + \frac{1}{\sqrt{N}}\sum_{j=1}^N U_j \int_{\r_+}\uno{z\leq f(X^N_{t-})}d\tilde \pi_j(t,z) + \ll(\frac1{\sqrt{N}}\sum_{j=1}^N U_j\rr) f(X^N_t)dt,
\end{equation}
where $\tilde\pi_j(dt,dz) := \pi_j(dt,dz)-dtdz$ is the compensated Poisson measure of~$\pi_j$ ($1\leq j\leq N$). Under this form, the second term of the SDE above is a local martingale (conditionally on the environment) whose jump heights vanish as $N$ goes to infinity. So, in the limit equation, it creates the Brownian term in~\eqref{barXw}. And the third term in the SDE above is a drift term that corresponds to the second term of~\eqref{barXw} since, according to the CLT, the variable~$N^{-1/2}\sum_{j=1}^N U_j$ converges in distribution to~$W\sim\mathcal{N}(0,\sigma^2).$

The main result of the paper is the convergence in distribution of the process~$(X^N_t)_t$ to the process~$(\bar X_t)_t.$ Our approach only allows to prove the annealed convergence of the processes, because there is a tricky problem concerning the quenched result: for a convergence to be true in our model, we need to guarantee $N^{-1/2}\sum_{j=1}^N U_j$ to converge. {This convergence holds true in distribution but cannot be true almost surely. We still manage to prove some properties related to a quenched convergence.}

The dynamics~\eqref{XNu}$-$\eqref{XN} and~\eqref{barXw}$-$\eqref{barX} are similar to respectively equations~$(3.9)$ and~$(3.10)$ of \cite{pfaffelhuber_mean-field_2022}. Indeed, if we consider $b(x):=-\alpha x$ for the drift function in our model, we obtain exactly the same dynamics as \cite{pfaffelhuber_mean-field_2022} for the convolution kernel~$\phi(t) := e^{-\alpha t}.$ Consequently, our main result Theorem~\ref{mainresult} can be compared to Theorem~2 of \cite{pfaffelhuber_mean-field_2022}. However, the proof of \cite{pfaffelhuber_mean-field_2022} cannot be used to prove Theorem~\ref{mainresult} since it relies on the assumption that the environment variables~$U_j$ ($j\geq 1$) follow Rademacher distribution, whereas we only assume the environment variables to be centered with some exponential moments. Note that it would also not be possible to prove Theorem~2 of \cite{pfaffelhuber_mean-field_2022} with the proof of our Theorem~\ref{mainresult} since it requires to have some Markovian structure (conditionally to the environment variables), which is not the case for Hawkes processes with general convolution kernel as in \cite{pfaffelhuber_mean-field_2022}. Another interesting point of Theorem~\ref{mainresult} is that we have an explicit convergence speed.

The model of this paper is also close to the one of~\cite{erny_mean_2022}, except that in~\cite{erny_mean_2022} the environment variables $U_j$ ($1\leq j\leq N$) in~\eqref{XNu} were replaced by centered marks of the point processes~$(Z^{N,j}_t)_t$ defined as
$$Z^{N,j}_t := \int_{[0,t]\times\r_+}\uno{z\leq f(X^{N}_{s-})}d\pi_j(s,z).$$

In term of neurosciences, the model of \cite{erny_mean_2022}, contrary to the model~\eqref{XNu}$-$\eqref{XN} and \cite{pfaffelhuber_mean-field_2022}, is not consistent with the following biological property: the role of a synapse cannot change (i.e. it is either always excitatory or always inhibitory). Mathematically, if the rate function~$f$ is monotone, it means that every time some particle~$i$ interacts with another particle~$j$ through a synaptic strength~$U_i$, the sign of $U_i$ is always the same. This property is not guaranteed in \cite{erny_mean_2022}, where the~$U_i$ are marks of some point processes (and hence they change at each spiking time of the same neuron), but it holds true in the model of this paper where they are fixed environment variables.

The model of this paper is a priori harder to study because, in \cite{erny_mean_2022}, the processes~$(X^N_t)_t$ (and their limit) are Markov processes and semimartingales, which is not the case here because of the random environment. In order to apply results from the theories of Markov processes and semimartingales, we need to work conditionally on the environment and, in a second time, to integrate over the environment.

To work conditionally on the environment in a proper manner, we use a coupling between the variables $U_j$ ($j\geq 1$) of~\eqref{XN} and the Gaussian variable~$W$ of~\eqref{barX}, corresponding to a coupling result for the classical CLT. To be more precise and formal, we construct a sequence of identically distributed and non-independent Gaussian variables~$(W^{[N]})_N$ such that, for every~$N\geq 2,$
\begin{equation}\label{couplage}
\ll|\frac1{\sqrt{N}}\sum_{j=1}^N U_j - W^{[N]}\rr|\leq K\frac{\ln N}{\sqrt{N}},
\end{equation}
where $K>0$ is a random variable independent of~$N.$ This coupling and the control that we use on the random variable~$K$ rely on Theorem~1 of \cite{komlos_approximation_1976}. Indeed, if we assume that the distribution~$\mu$ admits some exponential moments: there exists~$\alpha>0$ such that
$$\int_\r e^{\alpha|x|}d\mu(x)<\infty,$$
then, it is possible (following the reasonning of Section~7.5 of \cite{ethier_markov_2005} that relies on Theorem~1 of \cite{komlos_approximation_1976}) to construct on the same probability space (possibly enlarged) as $(U_j)_{j\geq 1}$ a standard one-dimensional Brownian motion~$(\beta_t)_t$ such that, the random variable~$K$ defined as
$$K := \underset{N\geq 2}{\sup}~\ll|\sum_{j=1}^N U_j - \sigma\beta_N\rr|/\ln N$$
is finite almost surely and admits exponential moments (this is stated and proved in Lemma~\ref{momentK} for self-containedness). Then, defining $W^{[N]}$ in the following way
$$W^{[N]} := \sigma \beta_N/\sqrt{N}$$
gives exactly~\eqref{couplage}. This construction is recalled at Appendix~\ref{couplingproof}.

The coupling~\eqref{couplage} allows us to compare the conditional distributions of the processes~$(X^N_t)_t$ and~$(\bar X_t^N)_t$ given the environment variables (where $\bar X^N$ is defined as $\bar X(W^{[N]})$ in~\eqref{barX}). {This results can be used to prove that the difference of the finite-dimensional distributions vanishes almost surely w.r.t. the randomness of the environment. However, we cannot obtain a quenched convergence in the usual sense, since the version of the limit to which we compare $X^N$ is $\bar X^N$ which still depends on~$N$. More precisely, the problem is the fact that, the sequence of Gaussian variables~$W^{[N]}$ has (almost surely) many subsequential limits: indeed, recalling that $W^{[N]} := \sigma\beta_N/\sqrt{N}$ for some standard Brownian motion~$\beta,$ it is of common knowledge that the limit superior (resp. inferior) of $W^{[N]}$ is infinity (resp. minus infinity) almost surely. 

To prove this result, we need to introduce formally~$\E$ the sigma-field related to the random environment:
$$\E := \sigma\ll(U_j~:~j\geq 1\rr)\vee\sigma\ll(W^{[N]}~:~N\in\n^*\rr).$$

And, in order to compare the conditional distributions of $(X^N_t)_t$ and $(\bar X^N_t),$ we introduce their (conditional) semigroups $(P^N_{\E,t})_t$ and $(\bar P^N_{\E,t})_t$, and their infinitesimal generators $A^N_\E$ and $\bar A^N_\E$ w.r.t.~$\E$ (see Section~\ref{notation} for the definitions). Using our coupling, we obtain a bound for the difference of the generators for sufficiently smooth test-functions, and deduce a bound for the semigroups using the following formula (which can be found in Lemma~1.6.2 of \cite{ethier_markov_2005} and in equation~$(3.1)$ of \cite{erny_mean_2022} with different terminologies and hypotheses): for $g$ smooth enough, $t\geq 0$ and $x\in\r,$
\begin{equation}\label{trotter}
P^N_{\E,t}g(x) - \bar P^N_{\E,t}g(x) = \int_0^t P^N_{\E,t-s}\ll(A^N_\E - \bar A^N_\E\rr)\bar P^N_{\E,s}g(x)ds.
\end{equation}

Note that, to use the formula above, one needs to guarantee some regularity properties on the limit semigroup~$(\bar P^N_{\E,t})_t$. This is proved using the regularity of the stochastic flow of the process~$(\bar X^N_t)_t$ (with a proof similar as the one of Proposition~3.4 of \cite{erny_mean_2022}).

Let us finally mention that the same kind of model (i.e. particle systems directed by SDEs where the jump term depends on random environment variables) has already been studied in normalization~$N^{-1}$ in~\cite{chevallier_mean_2019} and more recently in~\cite{agathe-nerine_multivariate_2022}. In both of these references, the random environment relies on the spatial structure of the particle system. To the best of our knowledge, \cite{pfaffelhuber_mean-field_2022} is the first paper about the convergence of this type of model in normalization~$N^{-1/2}.$ However this kind of convergence concerning models where a drift term is driven by a random environment in normalization~$N^{-1/2}$ have already been proved: e.g. \cite{ben_arous_large_1995}, \cite{guionnet_averaged_1997} and \cite{dembo_universality_2021}. And, in a similar framework, \cite{lucon_quenched_2011} has proved a quenched convergence of the fluctuations of a particle system (with a drift term driven by a random environment) in normalization~$N^{-1},$ which is a similar regime. Note that the tricky problem concerning the quenched convergence mentionned in the previous paragraph (also related to Remark~\ref{noquench}) is also encountered in \cite{lucon_quenched_2011}.

{\bf Organization.} In Section~\ref{notation}, we introduce the notation that we use throughout the paper. The assumptions, the main result about the annealed converge (i.e. Theorem~\ref{mainresult}) and a quenched control (i.e. Proposition~\ref{quenched}) are stated in Section~\ref{assumpresult}. Section~\ref{proofmainresult} is dedicated to prove our main result Theorem~\ref{mainresult}, while Proposition~\ref{quenched} is proved in Section~\ref{proofquenched}. Finally, the CLT coupling is formally recalled and its main property is proved in Appendix~\ref{couplingproof}, and Appendix~\ref{prooftechnical} gathers the proofs of some technical lemmas.

\subsection{Notation}
\label{notation}

In the paper, we use the following notation:
\begin{itemize}
	\item For $T>0,$ we note $D([0,T],\r)$ (resp. $D(\r_+,\r)$) the set of c\`adl\`ag functions defined on~$[0,T]$ (resp. $\r_+$) endowed with Skorohod topology (see for example Section~12 (resp. Section~16) of \cite{billingsley_convergence_1999}).
	\item For $n\in\n^*,$ $C^n_b(\r)$ denotes the set of real-valued functions defined on~$\r$ that are $C^n$ such that all their derivatives (up to order~$n$) are bounded.
	\item For $n\in\n^*,$ and $g\in C^n_b(\r),$ we note
	$$||g||_{n,\infty} := \sum_{k=0}^n \ll|\ll| g^{(k)}\rr|\rr|_\infty.$$
	\item For $\nu_1,\nu_2$ distributions on~$\r$ with finite first order moments, we use the notation $d_{KR}(\nu_1,\nu_2)$ for the Kantorovich-Rubinstein metric between~$\nu_1$ and~$\nu_2.$ This quantity is defined as:
	$$d_{KR}(\nu_1,\nu_2) := \underset{g}{\sup}~\ll\{\int_\r g(x)d\nu_1(x) - \int_\r g(x) d\nu_2(x)\rr\},$$
	where the supremum is taken over the Lipschitz continuous functions~$g:\r\rightarrow\r$ whose Lipschitz constants are non-greater than one.
	\item If $(X_t)_{t\geq 0}$ is a real-valued Markov process, $(X^{(x)}_t)_{t\geq 0,x\in\r}$ denotes its stochastic flow. In other words, for $x\in\r,$ $(X^{(x)}_t)_t$ is the process starting at position~$x\in\r$ that is defined by the dynamics of~$(X_t)_{t\geq 0}.$ In addition, if the stochastic flow is~$n$ times differentiable w.r.t. the space variable~$x$ at time~$t$, we note $\partial^n_x X^{(x)}_t$ the related derivative.
	\item We note $U_j$ ($j\geq 1$) the environment variables introduced in~\eqref{XNu}$-$\eqref{XN} (see Assumption~\ref{hypenv} for the hypotheses satisfied by these random variables) and $\sigma^2$ their variance. The variables~$W^{[N]}$ ($N\geq 1$) coupled with the $U_j$ ($j\geq 1$) are introduced in~\eqref{couplage} (or alternatively at Appendix~\ref{couplingproof}). Each $W^{[N]}$ is used to define the version of the limit equation~\eqref{barXw}$-$\eqref{barX} where the environment is coupled with the one of the $N$-th equation~\eqref{XNu}$-$\eqref{XN}.
	\item In all the paper, we always note $(\Omega,\mathbb{P})$ the probability spaces on which we work (only the environments of the equations are coupled, so all the processes need not be defined on the same space). They are assumed to be complete and large enough to define what we need. They will not be given explicitly.
	\item We note $\E$ the sigma-field of the environment variables:
	$$\E := \sigma\ll(U_j~:~j\geq 1\rr)\vee\sigma\ll(W^{[N]}~:~N\in\n^*\rr).$$
	In addition, $\mathbb{E}_\E$ designates the conditional expectation given~$\E,$ and $\mathbb{P}_\E$ the related probability measure (i.e. $\mathbb{P}_\E(A) = \mathbb{E}_\E[1_{A}]$). If $\nu$ is a distribution on~$\r,$ we may note $\mathbb{E}_{\E,\nu}[h(X^N)]$ the expectation conditionally on~$\E$ under which $X^N_0$ has law~$\nu,$ and a similar notation for~$\bar X^N$ (notice that this notation makes sense because these processes are Markov processes conditionally on~$\E$).
	\item As stated earlier, $\nu^N_0$ and $\bar\nu_0$ are the respective distributions of $X^N_0$ and $\bar X_0.$ In the following, the quantity $\mathbb{E}_\E[g(X^N_t)]$ always denotes $\mathbb{E}_{\E,\nu^N_0}[g(X^N_t)]$ (with similar notation for the process~$\bar X$).  
	\item We note $(P^N_{\E,t})_t$ and $A^N_\E$ respectively the semigroup and the (extended) generator of the process~$(X^N_t)_t$ defined at~\eqref{XNu}$-$\eqref{XN}, conditionally on the environment~$\E.$ Similarly, we use the notation $(\bar P^N_{\E,t})_t$ and $\bar A^N_\E$ for the limit process~$\bar X^N_t$ defined at~\eqref{barXw}$-$\eqref{barX}, defined w.r.t. the environment variable $W := W^{[N]}.$ Note that it would not be possible to consider the non-conditional semigroups and generators since these processes are only Markovian conditionally on the environment. For the precise notion of semigroups and (extended) generators we use in this paper, we refer to Appendix~$A$ of \cite{erny_mean_2022}: for $g$ in $C_b(\r),$ $x\in\r$ and $t\geq 0,$
	$$P_{\E,t}g(x) := \espcc{\E}{g(X_t^{(x)})} = \espcc{\E,\delta_x}{g(X_t)},$$
	and, $A_\E g$ (if it exists) is characterized as the unique function such that for all $x\in\r,t\geq 0,$
	$$P_{\E,t}g(x) - g(x) = \int_0^t P_{\E,s}A_\E g(x)ds.$$
	\item {Given two positive real-valued sequences~$(u_N)_N$ and~$(v_N)_N$, we use the standard notation $u_N = \mathcal{O}(v_N)$ for: there exists~$N_0\in\n$ and~$K>0$ such that, for all~$N\geq N_0$, $u_N \leq K v_N$. Let us point out that, if the numbers~$u_N$ and~$v_N$ ($N\in\n$) are random variables, then the meaning of "almost surely, $u_N = \mathcal{O}(v_N)$" has to be understood as previously where~$N_0$ and~$K$ may be random quantities (not necessarily random variables).}
	\item We note $C$ any positive constant. The value of $C$ can change from line to line in an equation. And, if $C$ depends on some parameter~$\theta,$ we note $C_\theta$ instead (the dependency w.r.t. the parameters of the model such as $b,f,\mu,\sigma$ will not be written explicitly).
\end{itemize}

\subsection{Assumptions and main result}
\label{assumpresult}

Let us state the hypotheses we need to prove our main result. The first assumption concerns the environment variables $U_j$ ($j\geq 1$). It allows to use the result of \cite{komlos_approximation_1976} to couple the environment variables $U_j$ ($j\geq 1$) with a sequence of Gaussian variables~$W^{[N]}$ used to define the random environment of the limit equation (see Appendix~\ref{couplingproof} for the formal coupling).

\begin{assu}\label{hypenv}
	The variables $U_j$ ($j\geq 1$) are i.i.d. $\mu-$distributed. The law $\mu$ is centered and admits exponential moments: there exists some $\alpha>0$ such that
	$$\int_\r e^{\alpha|x|}d\mu(x)<\infty.$$
	We note $\sigma^2$ the variance of~$\mu.$
\end{assu}

The next assumption gives the conditions we need on the coefficients of the SDEs~\eqref{XNu} and~\eqref{barXw}. Let us note that the function~$f$ is obviously assumed to be non-negative in all the paper, so we will not write this condition in our assumptions.

\begin{assu}\label{hypfun}
	The functions~$b$,~$f$ and~$\sqrt{f}$ are $C^4$, and, for each $k\in\{1,2,3,4\},$ the functions $b^{(k)}$, $f^{(k)}$ and $\sqrt{f}^{(k)}$ are bounded.
\end{assu}

Note that Assumption~\ref{hypfun} implies that the functions~$b$,~$f$ and~$\sqrt{f}$ are Lipschitz continuous. In particular, under this condition, the equations~\eqref{XNu} and~\eqref{barXw} are (strongly) well-posed (it is a consequence of Theorem~IV.9.1 of \cite{ikeda_stochastic_1989} and of classical a priori estimates that can be established on the processes, conditionally on the environment: see for instance Proposition~$C.1$ of \cite{erny_mean_2022} or Proposition~2 of \cite{fournier_toy_2016}). Once the strong well-posedness is established given a fixed environment, it is possible to define an explicit probability space containing the random environment as well as what is needed in the SDEs (i.e. Poisson measures, Brownian motion and initial conditions) such that they are independent. Moreover Assumption~\ref{hypfun} is a necessary condition of Theorem~1.4.1 of \cite{kunita_lectures_1986} that implies that the stochastic flow of the limit process $(\bar X_t)_t$ is $C^3$ w.r.t. its initial condition (see also Theorem~4.6.5 of \cite{kunita_stochastic_1990}).

\begin{rem}
	In our proof, the only condition from Assumption~\ref{hypfun} that we use directly is that the functions~$b$ and~$f$ are measurable and sublinear. The stronger conditions allow to guarantee the well-posedness of our equations and to apply Theorem~1.4.1 of \cite{kunita_lectures_1986}.
\end{rem}

Now we give the hypotheses we need on the laws of the initial conditions of the processes~$(X^N_t)_t$ ($N\in\n^*$) and~$(\bar X_t)_t.$

\begin{assu}\label{hypinit}$ $
	\begin{enumerate}
		\item The probability measures $\nu_0^N$ converge to~$\bar\nu_0$ for the Kantorovich-Rubinstein metric~$d_{KR}.$ Equivalently, it means that $\nu_0^N$ converges weakly to $\bar\nu_0$ and that $\int |x|d\nu^N_0(x)$ converges to $\int |x|d\bar\nu_0(x)$ as $N$ goes to infinity (see Theorem~6.9 of \cite{villani_optimal_2008}).
		\item The probability measures~$\nu_0^N$ admit sixth order moments uniformly bounded in~$N\in\n^*,$ and $\bar\nu_0$ admits a first order moment:
		$$\underset{N\in\n^*}{\sup}\int_\r x^6 d\nu_0^N(x)<\infty\textrm{ and }\int_\r |x|d\bar\nu_0(x)<\infty.$$
	\end{enumerate}
\end{assu}

Since our main result is the convergence in distribution of $(X^N_t)_t$ to $(\bar X_t)_t,$ we obviously need an assumption on the convergence of their initial conditions. Item~1 of Assumption~\ref{hypinit} is a convergence a bit stronger than the necessary convergence in distribution, but it seems to be the most appropriate in our proof. Besides, we need to assume that the initial conditions of the processes admit finite first order moments in order to manipulate Kantorovich-Rubinstein metric. And, for technical reasons, we also need some a priori estimates on the process~$(X^N_t)_t$ that are uniform in~$N\in\n^*.$

{Now we state our main result: the annealed convergence in distribution of the processes $(X^N_t)_t$ as $N$ goes to infinity, with an explicit convergence speed for their finite-dimensional distributions.}

\begin{thm}\label{mainresult}
	Grant Assumptions~\ref{hypenv},~\ref{hypfun} and~\ref{hypinit}. The process~$(X^N_t)_t$ converges to~$(\bar X_t)_t$ in distribution on~$D(\r_+,\r).$ In addition, we have the following convergence speed for the finite-dimensional distributions: for all $T>0$, there exists some $N_T\in\n^*$ such that, for any $k\in\n^*,t_1\leq ...\leq t_k\leq T,$ $g_1,...,g_k\in C^3_b(\r)$ and $N\geq N_T,$
	$$\ll|\esp{g_1(X^N_{t_1})...g_k(X^N_{t_k})} - \esp{g_1(\bar X_{t_1})...g_k(\bar X_{t_k})}\rr|\leq C_{T,k,g_1,...,g_k}\ll(d_{KR}(\nu_0^N,\bar\nu_0) +\ln N/\sqrt{N}\rr),$$
	for some positive constant $C_{T,k,g_1,...,g_k}.$
\end{thm}

The proof of the convergence (and the bound for the convergence speed) of the finite-dimensional distributions of Theorem~\ref{mainresult} is given in Section~\ref{sectionCVfidi}. The (annealed) convergence in distribution of the processes is deduced from it in Section~\ref{sectiontight}.

{
In the proof of Theorem~\ref{mainresult}, we compare the conditional distributions (given the randomness of the environment~$\mathcal{E}$) of~$X^N$ and~$\bar X^N$. This is made possible by the fact that the random environments of the processes are coupled. This type of approach could be used to obtain a quenched convergence of the distributions instead of just an annealed one. However the quenched convergence of the sequence of processes~$(X^N)_N$ cannot hold in our model: we approximate the conditional distribution of~$X^N$ with the one of $\bar X^N$, which still depends on~$N$ (trajectorially w.r.t.~$\mathcal{E}$). The formal problem to obtain a quenched convergence is the a.s. divergence of the sequence of Gaussian variables~$W^{[N]}$. 

However, the aim of the next proposition is to guarantee some "quenched control" between the conditional finite-dimensional distributions of~$X^N$ and~$\bar X^N$

\begin{prop}\label{quenched}
	Grant Assumptions~\ref{hypenv},~\ref{hypfun} and~\ref{hypinit}. The following statement holds true almost surely conditionally on~$\mathcal{E}$:
	 for any $k\in\n^*,T>0,$ $t_1\leq ...\leq t_k\leq T,$ and $g_1,...,g_k\in C^3_b(\r),$
	$$\ll|\espcc{\E}{g_1(X^N_{t_1})...g_k(X^N_{t_k})} - \espcc{\E}{g_1(\bar X^N_{t_1})...g_k(\bar X^N_{t_k})}\rr| = \mathcal{O}\left((\ln N)^{C_{T,k}}\ll[d_{KR}(\nu_0^N,\bar\nu_0) +N^{-1/2}\rr]\right),$$
	with $C_{T,k}$ some positive and deterministic constant that depends only on~$T,k$ and the model parameters.
\end{prop}

The proof of Proposition~\ref{quenched} is given at Section~\ref{proofquenched}.

Note that, according to Proposition~\ref{quenched}, to have the a.s. vanishing of the difference of the finite-dimensional distributions on any compact interval between the processes~$X^N$ and~$\bar X^N$, one needs to strengthen Assumption~\ref{hypinit}$.1$ by assuming that, for any positive constant~$C>0$,
$$(\ln N)^C d_{KR}(\nu_0^N,\bar\nu_0)\underset{N\rightarrow\infty}{\longrightarrow}0.$$
which can be guaranteed, for example, by imposing that the Kantorovich-Rubinstein metric between $\nu_0^N$ and $\bar\nu_0$ is non-greater that $N^{-\eps}$ for some fixed~$\eps>0$.

{
The conditional tightness of the processes $X^N$ cannot be true in general: in~\eqref{XNheuristique} the two first terms are bounded in $L^p$ (a.s., conditionally on the environment), while the third one is unbounded such that its limits superior and inferior are infinity and minus infinity. In particular, the quenched convergence of~$X^N$ cannot be guaranteed in general. Obtaining a kind of quenched convergence is possible along random subsequences, whose randomness depends only on~$\E$~: for any $\lambda\in\r,$ one can consider an increasing sequence~$(\phi_\lambda(N))_N$ such that
$$W^{[\phi_\lambda(N)]}\underset{N\rightarrow\infty}{\longrightarrow}\lambda~~a.s.$$
which implies, according to semimartingales convergence results (like Theorem~IX.3.48 of \cite{jacod_limit_2003}) applied conditionally on~$\E$, the quenched convergence of~$X^{\phi_\lambda(N)}$ to the process~$\bar X(\lambda)$ defined at~\eqref{barXw}.
}

%

%

\section{Proof of Theorem~\ref{mainresult}: annealed convergence of $(X^N)_N$}
\label{proofmainresult}

In Section~\ref{sectionapriori}, we state a lemma about a priori estimates for the process~$(X^N_t)_t$ that are uniform in~$N$. The fact that we work in a random environment and in normalization~$N^{-1/2}$ makes the proof more tricky than in more usual frameworks. Then, Section~\ref{sectionCVfidi} is dedicated to prove the finite-dimensional convergence in distribution of $(X^N_t)_t$ to $(\bar X_t)_t.$ As explained in the introduction, we begin by controlling the difference of the conditional generators of $X^N$ and $\bar X^N$ to deduce a bound for the difference of the conditional semigroups using~\eqref{trotter}. This allows to control the difference of the conditional one-dimensional distributions of $X^N$ and $\bar X,$ then the control between the conditional finite-dimensional distributions is deduced using usual Markov processes technics. Afterwards, the (non-conditional) finite-dimensional convergence in distribution is obtained by integrating over the environment in the control previously obtained. Finally in Section~\ref{sectiontight}, we deduce the convergence in distribution of the processes from the convergence of their finite-dimensional distributions using Theorem~13.5 of \cite{billingsley_convergence_1999}. Section~\ref{sectionproofmain} is a mere summary of the results of the two previous sections.

\subsection{A priori estimates}\label{sectionapriori}

Since we work in a random environment, we need, in a first time, to work conditionally on the environment to obtain some a priori estimates. However, in our framework, we need to take precautions to prevent the problematic term $N^{-1/2}\sum_{j=1}^N|U_j|$ from appearing, and to control instead $|N^{-1/2}\sum_{j=1}^N U_j|$. In addition, we need to apply Gr\"onwall's lemma conditionally on the environment, which makes arise the previous term within an exponential. The next lemma (whose proof is postponed to Appendix~\ref{sectionsomme}) provides a suitable control for this exponential.

\begin{lem}\label{momentsomme}
	Let $U_i$ ($i\geq 1$) be i.i.d. centered random variables having some finite exponential moments. Let $S_N$ be defined as
	$$S_N := \frac{1}{\sqrt{N}}\sum_{i=1}^N U_i.$$
	
	Then, for all $\gamma>0,$ there exists some $N_\gamma\in\n^*$ such that,
	$$\underset{N\geq N_\gamma}{\sup}~\esp{e^{\gamma|S_N|}} <\infty.$$
	
	As a consequence, for any $p\in\n^*,$ 
	$$\underset{N\geq 1}{\sup}~\esp{|S_N|^p}<\infty.$$
\end{lem}

\begin{rem}\label{remmomentsomme}
	The second part of Lemma~\ref{momentsomme} is easy to prove with straightforward computation. In the proof of Lemma~\ref{momentsomme}, we prefer to use the first part of the lemma to prove the second one, since it is a straightforward consequence. Note also that it is necessary to consider some $N_\gamma$ in the first statement (but not in the second one) since $S_1 := U_1$ is not assumed to admit all exponential moments.
\end{rem}

The next lemma states the a priori estimates we need to prove our main result. Its proof is postponed to Appendix~\ref{sectionaprioriproof}.

\begin{lem}\label{apriori}
	Grant Assumption~\ref{hypenv}. Assume that $b$ and $f$ are measurable and sublinear. Then, any solution~$(X^N_t)_t$ of~\eqref{XNu}$-$\eqref{XN} satisfies a priori the following: for every $p\in\n^*,$ if $X^N_0$ admits a finite $2p-$order moment, then for all $T>0,$ there exists some $N_{T,p}\in\n^*$ such that
	\begin{itemize}
		\item[$(i)$] $$\underset{N\geq N_{T,p}}{\sup}~\underset{t\leq T}{\sup}~\esp{\ll(X^N_t\rr)^{2p}}\leq C_{T,p}\ll(1 + \esp{(X^N_0)^{2p}}\rr),$$
		\item[$(ii)$] and, if $p\geq 2,$ then for any $2\leq \kappa <2p,$ $$\underset{N\geq N_{T,p}}{\sup}~\esp{\underset{t\leq T}{\sup}~\ll|X^N_t\rr|^{\kappa}}<\infty.$$
	\end{itemize}
\end{lem}

\subsection{Convergence of the finite-dimensional distributions}\label{sectionCVfidi}

The goal of this section is to prove the convergence speed stated in Theorem~\ref{mainresult}, namely: for any $k\in\n^*,T>0,t_1\leq ...\leq t_k\leq T,$ and $g_1,...,g_k\in C^3_b(\r),$
\begin{equation}\label{CVfidi}
\ll|\esp{g_1(X^N_{t_1})...g_k(X^N_{t_k})} - \esp{g_1(\bar X_{t_1})...g_k(\bar X_{t_k})}\rr|\leq C_{T,k,g_1,...,g_k}\ll(d_{KR}(\nu_0^N,\bar\nu_0) + \ln N/\sqrt{N}\rr),
\end{equation}
for some positive constant $C_{T,k,g_1,...,g_k}.$

Before proving~\eqref{CVfidi}, let us state two useful Lemmas. The first one allows to prove that the limit conditional semigroup is $C^3$ and gives a control of its derivatives. Its proof is postponed to Appendix~\ref{proofcontrolsemi}.

\begin{lem}\label{controlsemigroup}
	Grant Assumption~\ref{hypfun}. Then, for all $t\geq 0,$ the function $x\in\r\mapsto \bar P_{\E,t}g(x)$ belongs to $C^3_b(\r)$ and
	$$\ll|\ll|\bar P_{\E,t}g\rr|\rr|_{3,\infty}\leq C_t ||g||_{3,\infty}e^{C_t|W|},$$
	where $W$ is the Gaussian variable in~\eqref{barXw}$-$\eqref{barX}.
\end{lem}

\begin{rem}
	In the statement of Lemma~\ref{controlsemigroup}, we use the version of the limit process~$(\bar X_t)_t$ with a generic Gaussian variable~$W$ for the sake of notation. However, in the proof below, we use this lemma with the version of the limit process~$(\bar X^N_t)_t$ defined w.r.t. the Gaussian variables $W^{[N]}$ ($N\in\n^*$).
\end{rem}

The next lemma recalls the formula~\eqref{trotter}. It permits to deduce the convergence of the (conditional) semigroups from the convergence of their (conditional) generators. Since in the formula we condition by the environment, the proof of this lemma is exactly the same as the one of Proposition~$B.2$ of \cite{erny_mean_2022}. This proof is therefore omitted

\begin{lem}\label{lemmetrotter}
	If Assumption~\ref{hypfun} holds true, then for all $t\geq 0,x\in\r$ and~$g\in C^3_b(\r),$
	$$P^N_{\E,t}g(x) - \bar P^N_{\E,t}g(x) = \int_0^t P^N_{\E,t-s}\ll(A^N_\E - \bar A^N_\E\rr)\bar P^N_{\E,s}g(x)ds.$$
\end{lem}

Using the two previous lemmas, we obtain the covergence speed of the finite-dimensional distribution of Theorem~\ref{mainresult}.
\begin{proof}[Proof of~\eqref{CVfidi}]
	{\it Step~1.} In this first step, we study the convergence of the one-dimensional distributions of the process $(X^N_t)_t$ to the limit process~$(\bar X^N_t)_t$, conditionally on the environment~$\E.$ Recall that conditionally on~$\E,$ these processes are Markov processes, hence their conditional one-dimensional distributions are characterized by the (conditional) semigroups of the processes. We note
	$$P^N_{\E,t}g(x) := \espcc{\E,\delta_x}{g(X^N_t)}\textrm{ and }\bar P^N_{\E,t}g(x):=\espcc{\E,\delta_x}{g(\bar X^N_t)}$$
	these semigroups.
	
	We have for any $g\in C^1_b(\r),x\in\r,$
	$$A^N_\E g(x) = b(x)g'(x) + f(x)\sum_{j=1}^N\ll[g\ll(x+\frac{U_j}{\sqrt{N}}\rr) - g(x)\rr],$$
	
	and for $g\in C^2_b(\r),x\in\r,$
	$$\bar A^N_\E g(x) = b(x)g'(x) + W^{[N]}f(x)g'(x) + \frac12 \sigma^2 f(x) g''(x).$$
	
	As a consequence, for all $g\in C^3_b(\r),x\in\r,$
	\begin{align*}
	\ll|A^N_\E g(x) - \bar A^N_\E g(x)\rr| \leq & f(x)\sum_{j=1}^N\ll|g\ll(x+\frac{U_j}{\sqrt{N}}\rr) - g(x) - \frac{U_j}{\sqrt{N}}g'(x) - \frac{U_j^2}{2N}g''(x)\rr|\\
	&+f(x)g'(x)\ll|\frac1{\sqrt{N}}\sum_{j=1}^N U_j - W^{[N]}\rr| + \frac12 f(x)g''(x)\ll|\frac1N\sum_{j=1}^N U_j^2 - \sigma^2\rr|\\
	\leq& f(x)||g||_{3,\infty}\ll(\frac{1}{6N\sqrt{N}} \sum_{j=1}^N|U_j|^3 + K\frac{\ln N}{\sqrt{N}} + \frac12\ll|\frac1N\sum_{j=1}^N U_j^2 - \sigma^2\rr|\rr).
	\end{align*}
	
	By Lemma~\ref{lemmetrotter}, for any $g\in C^3_b(\r),x\in\r,t\geq 0,$
	$$P^N_{\E,t}g(x) - \bar P^N_{\E,t}g(x) = \int_0^t P^N_{\E,t-s}\ll(A^N_\E - \bar A^N_\E\rr)\bar P^N_{\E,s}g(x)ds,$$
	whence
	\begin{multline*}
	\ll|P^N_{\E,t}g(x) - \bar P^N_{\E,t}g(x)\rr|\leq \int_0^t \espcc{\E,x}{\ll|A^N_\E \bar P^N_{\E,s}g(X^N_{t-s}) - \bar A^N_\E\bar P^N_{\E,s}g(X^N_{t-s})\rr|}ds\\
	\leq \int_0^t\ll|\ll| P^N_{\E,s}g\rr|\rr|_{3,\infty} \espcc{\E,x}{f(X^N_{t-s})}\ll(\frac{1}{6N\sqrt{N}}\sum_{j=1}^N|U_j|^3  + K\frac{\ln N}{\sqrt{N}} + \frac12\ll|\frac1N\sum_{j=1}^N U_j^2 - \sigma^2\rr|\rr)ds.
	\end{multline*}
	
	Using Lemma~\ref{controlsemigroup},
	\begin{align}
	&\ll|P^N_{\E,t}g(x) - \bar P^N_{\E,t}g(x)\rr|\nonumber\\
	&\leq C_t\ll(\int_0^t \espcc{\E,\delta_x}{f(X^N_s)}ds\rr) e^{C_t|W^{[N]}|}||g||_{3,\infty}\ll(\frac1{N\sqrt{N}}\sum_{j=1}^N|U_j|^3 + K\frac{\ln N}{\sqrt{N}} + \frac12\ll|\frac1N\sum_{j=1}^N U_j^2 - \sigma^2\rr|\rr)\nonumber\\
	&=:||g||_{3,\infty}\eps_N(t,x).\label{diffsemigroup}
	\end{align}
	
	{\it Step~2.} Let us now show that the control between the semigroups above can be extended to any finite-dimensional distributions of the processes, conditionally on the environment. The result of this second step (which is formally~\eqref{quenchedfidi}), will be used in the proof of Proposition~\ref{quenched} later.
	

	To begin with
	\begin{multline*}
	\ll|\espcc{\E}{g_1(X^N_{t_1})...g_k(X^N_{t_k})} - \espcc{\E}{g_1(\bar X^N_{t_1})...g_k(\bar X^N_{t_k})}\rr| \\
	= \ll|\espcc{\E,\nu^N_0}{g_1(X^N_{t_1})...g_k(X^N_{t_k})} - \espcc{\E,\bar\nu_0}{g_1(\bar X^N_{t_1})...g_k(\bar X^N_{t_k})}\rr|\\
	\leq \ll|\espcc{\E,\nu^N_0}{g_1(X^N_{t_1})...g_k(X^N_{t_k})} - \espcc{\E,\nu^N_0}{g_1(\bar X^N_{t_1})...g_k(\bar X^N_{t_k})}\rr|\\
	+\ll|\espcc{\E,\nu^N_0}{g_1(\bar X^N_{t_1})...g_k(\bar X^N_{t_k})} - \espcc{\E,\bar \nu_0}{g_1(\bar X^N_{t_1})...g_k(\bar X^N_{t_k})}\rr|=: A_1 + A_2.
	\end{multline*}
	
	Recall that any Markov process~$(X_t)_t$ with semigroup $(P_t)_t$ satisfies for all $t_1\leq t_2\leq...\leq t_k,$ and for all continuous and bounded functions $g_1,...,g_k,$
	\begin{multline*}
	\espcc{\nu}{g_1(X_{t_1})...g_{k-1}(X_{t_{k-1}}) g_k(X_{t_k})} \\
	=\int \nu(dx) \int P_{t_1}(x,dx_1)g(x_1)...\int P_{t_k-t_{k-1}}(x_{k-1},dx_k) g_k(x_k)\\
	=\int \nu(dx) \int P_{t_1}(x,dx_1)g(x_1)...\int P_{t_{k-1}-t_{k-2}}(x_{k-2},dx_{k-1}) \ll(g_{k-1} \cdot P_{t_k - t_{k-1}} g_k\rr)(x_{k-1})\\
	= \espcc{\nu}{g_1(X_{t_1})...g_{k-1}(X_{t_{k-1}}) P_{t_k-t_{k-1}}g_k(X_{t_{k-1}})},
	\end{multline*}
	
	Let us show by induction on~$k$ that
	\begin{equation}\label{A1}
	A_1\leq ||g_1||_{3,\infty}...||g_k||_{3,\infty} C_{T,k}e^{C_{T,k}|W_{[N]}|}\sum_{l=0}^{k-1}\espcc{\E}{\eps_N(T,X^N_{t_l})},
	\end{equation}
	recalling that $T\geq t_k>t_{k-1}>...>t_1.$
	
	The bound of~\eqref{A1} above for $k=1$ is valid by definition of~$\eps_N(T,x)$ (see~\eqref{diffsemigroup}). To show the inductive step, let us write
	\begin{align*}
	A_1=&\ll|\espcc{\E,\nu^N_0}{g_1(X^N_{t_1})...(g_{k-1}\cdot P^N_{t_k - t_{k-1}}g_k)(X^N_{t_{k-1}})} - \espcc{\E,\nu^N_0}{g_1(\bar X^N_{t_1})...(g_{k-1}\bar P^N_{t_k - t_{k-1}})(\bar X^N_{t_{k-1}})}\rr|\\
	\leq& \ll|\espcc{\E,\nu^N_0}{g_1(X^N_{t_1})...(g_{k-1}\cdot P^N_{t_k - t_{k-1}}g_k)(X^N_{t_{k-1}})} - \espcc{\E,\nu^N_0}{g_1(X^N_{t_1})...(g_{k-1}\cdot \bar P^N_{t_k - t_{k-1}}g_k)(X^N_{t_{k-1}})}\rr|\\
	&+\ll|\espcc{\E,\nu^N_0}{g_1(X^N_{t_1})...(g_{k-1}\cdot \bar P^N_{t_k - t_{k-1}}g_k)(X^N_{t_{k-1}})} - \espcc{\E,\nu^N_0}{g_1(\bar X^N_{t_1})...(g_{k-1}\bar P^N_{t_k - t_{k-1}})(\bar X^N_{t_{k-1}})}\rr|\\
	=:&A_{11} + A_{12},
	\end{align*}
	
	Then, thanks to~\eqref{diffsemigroup},
	\begin{align*}
	A_{11}\leq& ||g_1||_\infty...||g_{k-1}||_\infty\espcc{\E,\nu^N_0}{\ll|P^N_{t_k - t_{k-1}}g_k(X^N_{t_{k-1}}) - \bar P^N_{t_k - t_{k-1}}g_k(X^N_{t_{k-1}})\rr|}\\
	\leq& ||g_1||_{3,\infty}...||g_{k}||_{3,\infty}\espcc{\E,\nu^N_0}{\eps_N(T,X^N_{t_{k-1}})}.
	\end{align*}
	
	On the other hand, using the inductive hypothesis we have
	$$A_{12}\leq ||g_1||_{3,\infty}...||g_{k-1}||_{3,\infty}\ll|\ll|\bar P^N_{t_k - t_{k-1}}g_k\rr|\rr|_{3,\infty} C_{T,k-1} e^{C_{T,k-1} |W_{[N]}|}\sum_{l=0}^{k-2}\espcc{\E}{\eps_N(T,X^N_{t_l})},$$
	whence, by Lemma~\ref{controlsemigroup},
	$$A_{12}\leq ||g_1||_{3,\infty}...||g_{k}||_{3,\infty} C_{T,k} e^{C_{T,k} |W_{[N]}|}\sum_{l=0}^{k-2}\espcc{\E}{\eps_N(T,X^N_{t_l})}.$$
	
	Then recalling that $A_1 = A_{11} + A_{12}$, we have proved~\eqref{A1}.
	
	Now let us show that
	\begin{equation}\label{A2}
	A_2 \leq ||g_1'||_\infty ... ||g_k'||_\infty C_{T,k} e^{C_{T,k}|W_{[N]}|} d_{KR}(\nu^N_0,\bar\nu_0),
	\end{equation}
	for some constant~$C_k>0.$
	
	Let us define
	$$h^k_{g_1,...,g_k}(x) := \int \bar P^N_{\E,t_1}(x,dx_1) g_1(x_1)...\int \bar P^N_{\E,t_k - t_{k-1}} (x_{k-1},dx_k) g_k(x_k).$$
	
	By definition of Kantorovich-Rubinstein metric $d_{KR},$ to show that~\eqref{A2} holds true, it is sufficient to prove that
	\begin{equation}\label{hk}
	\ll|\ll|\ll(h^k_{g_1,...,g_k}\rr)'\rr|\rr|_\infty \leq ||g_1'||_\infty ... ||g_k'||_\infty C_{T,k} e^{C_{T,k}|W^{[N]}|}.
	\end{equation}
	
	Now let us prove~\eqref{hk} by induction on~$k.$ Since
	$$h^1_{g_1}(x):= \bar P^N_{\E,t_1} g_1(x),$$
	the case $k=1$ is a consequence of Lemma~\ref{controlsemigroup}. To show the inductive step, let us write
	\begin{align*}
	h^k_{g_1,...,g_k}(x) =& \int \bar P^N_{\E,t_1}(x,dx_1) g_1(x_1)...\int \bar P^N_{\E,t_k - t_{k-1}} (x_{k-1},dx_k) g_k(x_k)\\
	=& \int \bar P^N_{\E,t_1}(x,dx_1) g_1(x_1)...\int \bar P^N_{\E,t_{k-1} - t_{k-2}}(x_{k-2},dx_{k-1}) \ll(g_{k-1}\cdot \bar P^N_{\E,t_k-t_{k-1}}g_k\rr)(x_{k-1})\\
	=& h^{k-1}_{g_1,...,g_{k-2},g_{k-1}\cdot \bar P^N_{\E,t_k-t_{k-1}}g_k}(x).
	\end{align*}
	
	By induction hypothesis, and thanks to Lemma~\ref{controlsemigroup},
	\begin{align*}
	\ll|\ll|\ll(h^k_{g_1,...,g_k}\rr)'\rr|\rr|_\infty \leq& ||g'_1||_\infty...||g'_{k-1}||_\infty || (\bar P^N_{\E,t_k-t_{k-1}}g_k)'||_\infty C_{T,k-1}e^{C_{T,k-1}|W^{[N]}|}\\
	\leq& ||g'_1||_\infty...||g'_{k}||_\infty C_{T,k}e^{C_{T,k}|W^{[N]}|}.
	\end{align*}
	
	Consequently~\eqref{hk} holds true.
	
	So we have proved~\eqref{A1} and~\eqref{A2}. This implies that
	
	\begin{multline}\label{quenchedfidi}
	\ll|\espcc{\E}{g_1(X^N_{t_1})...g_k(X^N_{t_k})} - \espcc{\E}{g_1(\bar X^N_{t_1})...g_k(\bar X^N_{t_k})}\rr| \\
	\leq ||g_1||_{3,\infty}...||g_k||_{3,\infty} C_{T,k} e^{C_{T,k}|W_{[N]}|}\ll(d_{KR}(\nu^N_0,\bar\nu_0) + \sum_{l=0}^{k-1}\espcc{\E}{\eps_N(T,X^N_{t_l})}\rr).
	\end{multline}
	
	{\it Step~3.} Now we prove~\eqref{CVfidi} by integrating over the environment in the inequality above.
	\begin{multline*}
	\ll|\esp{g_1(X^N_{t_1})...g_k(X^N_{t_k})} - \esp{g_1(\bar X^N_{t_1})...g_k(\bar X^N_{t_k})}\rr|\\
	= \ll|\esp{\espcc{\E}{g_1(X^N_{t_1})...g_k(X^N_{t_k})} - \espcc{\E}{g_1(\bar X^N_{t_1})...g_k(\bar X^N_{t_k})}}\rr|\\
	\leq \esp{\ll|\espcc{\E}{g_1(X^N_{t_1})...g_k(X^N_{t_k})} - \espcc{\E}{g_1(\bar X^N_{t_1})...g_k(\bar X^N_{t_k})}\rr|}\\
	\leq ||g_1||_{3,\infty}...||g_k||_{3,\infty} C_{T,k} \esp{e^{2C_{T,k}|W_{[N]}|}}^{1/2}\ll(d_{KR}(\nu^N_0,\bar\nu_0) + \esp{\ll(\sum_{l=0}^{k-1}\espcc{\E}{\eps_N(T,X^N_{t_l})}\rr)^2}^{1/2}\rr)\\
	\leq C_{T,k} ||g_1||_{3,\infty}...||g_k||_{3,\infty}\ll(d_{KR}(\nu^N_0,\bar\nu_0) + \sqrt{k}\sum_{l=0}^{k-1}\sqrt{\esp{\eps_N(T,X^N_{t_l})^2}}\rr),
	\end{multline*}
	where we have used Cauchy-Schwarz' inequality in the line before last above, and in the last line, Jensen's inequality as well as the facts that for all $x,y\geq 0,$ $(x+y)^2\leq x^2/2 + y^2/2$ and $\sqrt{x+y}\leq\sqrt{x}+\sqrt{y}.$
	
	Recalling that the expression of~$\eps_N(T,x)$ is given explicitly in~\eqref{diffsemigroup}, applying Cauchy-Schwarz' inequality twice and recalling that $f$ is sublinear, we have
	\begin{multline*}
	\esp{\eps_N(T,X^N_{t_l})^2}\leq C_T \ll(1 + \underset{s\leq T}{\sup}~\esp{(X^N_s)^4}^{1/2}\rr) \esp{e^{8C_T|W_{[N]}|}}^{1/4}\\
	\times\ll(\esp{\ll(\frac{1}{N\sqrt{N}}\sum_{j=1}^N |U_j|^3\rr)^8}^{1/4} + \esp{\ll(K\frac{\ln N}{\sqrt{N}}\rr)^8}^{1/4} + \esp{\ll|\frac1N\sum_{j=1}^N U_j^2 - \sigma^2\rr|^8}^{1/4}\rr).
	\end{multline*}
	
	Then, by Jensen's inequality,
	$$\esp{\ll(\frac1{N\sqrt{N}}\sum_{j=1}^N|U_j|^3\rr)^8} = \esp{\frac1{N^4}\ll(\frac1N\sum_{j=1}^N|U_j|^3\rr)^8} \leq \esp{\frac{1}{N^4}\cdot \frac1N\sum_{j=1}^N |U_j|^{24}}= \frac1{N^4}\esp{|U_1|^{24}}.$$
	Now, let $V_j = U_j^2-\sigma^2$ ($j\geq 1$). Since the variables $V_j$ ($j\geq 1$) are i.i.d. and centered,
	$$\esp{\ll|\frac1N\sum_{j=1}^N U_j^2 - \sigma^2\rr|^8} = \esp{\ll(\frac1N\sum_{j=1}^N V_j\rr)^8} = \frac{1}{N^4}\esp{\ll(\frac{1}{\sqrt{N}}\sum_{j=1}^N V_j\rr)^8}\leq \frac1{N^4}C,$$
	by the second statement of Lemma~\ref{momentsomme} applied to the variables~$V_j$ ($j\geq 1$). Indeed, even if these variables are not assumed to admit some exponential moments, the second statement of Lemma~\ref{momentsomme} still hold true (see Remark~\ref{remmomentsomme}).
	
	Then, recalling that $K$ and $W^{[N]}\sim\mathcal{N}(0,\sigma^2)$ admit exponential moments and a fortiori every polynomial moment, and using Lemma~\ref{apriori},
	
	$$\esp{\eps_N(T,X^N_{t_l})^2}\leq C_T \ll(\frac1N + \ll(\frac{\ln N}{\sqrt{N}}\rr)^2 + \frac1N\rr)\leq C_T \ll(\frac{\ln N}{\sqrt{N}}\rr)^2.$$
	
	Consequently,
	$$\ll|\esp{g_1(X^N_{t_1})...g_k(X^N_{t_k})} - \esp{g_1(\bar X^N_{t_1})...g_k(\bar X^N_{t_k})}\rr|\leq C_{T,k} ||g_1||_{3,\infty}...||g_k||_{3,\infty}\ll(d_{KR}(\nu^N_0,\bar\nu_0) + \frac{\ln N}{\sqrt{N}}\rr),$$
	which gives~\eqref{CVfidi}.
\end{proof}

\subsection{Convergence of the processes in distribution}\label{sectiontight}

Since the finite-dimensional convergence in distribution of $(X^N_t)_t$ to $(\bar X_t)_t$ has been established in the previous section, it only remains to prove the tightness of the processes to obtain the convergence in distribution in Skorohod space. Instead of using the classical Aldous' criterion we rather rely on Theorem~$13.5$ of \cite{billingsley_convergence_1999} since it seems more appriopriate for technical reasons. Indeeed, notice that the conditional tightness of the sequence~$(X^N)_N$ cannot hold true because of the almost sure unboundedness of the sequence $(N^{-1/2}\sum_{j=1}^N U_j)_N$. So we would need to work on stopping times adapted to the (non-conditional) filtration of the processes. In particular, these stopping times would depend on the random environment in a non-trivial way, which seems hard to handle for the computation which require to work conditionally on the environment.

The following convergence criterion of Theorem~$13.5$ (where condition~$(13.13)$ is replaced by the stronger~$(13.14)$) appears to be more suitable as no stopping times are involved: for any $T>0,$
\begin{itemize}
	\item[$(i)$] $\bar X_T - \bar X_{T-\delta}$ converges to zero in law as $\delta$ vanishes, 
	\item[$(ii)$] for all $r\leq s\leq t\leq T, N\in\n^*,$
	$$\esp{\ll(X^N_s - X^N_r\rr)^2\ll(X^N_t - X^N_s\rr)^2}\leq C_T \ll(t-r\rr)^{3/2},$$
	for some $C_T>0$.
\end{itemize}

Item~$(i)$ is trivially satisfied since the trajectories of $\bar X$ are continuous.

The rest of this section is dedicated to prove Item~$(ii).$ For this purpose, let us rewrite the dynamics of $X^N$ in the following way
\begin{align*}
X^N_t =& X^N_0 + \int_0^t b(X^N_s)ds + \frac1{\sqrt{N}}\sum_{j=1}^NU_j\int_{[0,t]\times\r_+}\uno{z\leq f(X^N_{s-})}d\tilde\pi^j(s,z) + \frac{1}{\sqrt{N}}\sum_{j=1}^N U_j\int_0^tf(X^N_s)ds,
\end{align*}
where $\tilde\pi^j(ds,dz) := \pi^j(ds,dz) - dsdz$ is the compensated Poisson measure of~$\pi^j.$ For any $s\leq t\leq T,$ we have
\begin{align*}
\ll(X^N_t-X^N_s\rr)^2\leq& C(t-s)\int_s^t b(X^N_r)^2 dr + C\frac1{N}\ll|\sum_{j=1}^N U_j\int_{]s,t]\times\r_+}\uno{z\leq f(X^N_{r-})}d\tilde\pi^j(r,z)\rr|^{2}\\
&+ C\ll|\frac{1}{\sqrt{N}}\sum_{j=1}^N U_j\rr|^{2}(t-s)\int_s^tf(X^N_r)^2 dr\\
&=: B^N_{s,t} + M^N_{s,t} + F^N_{s,t}.
\end{align*}

Since $f$ and $b$ are sublinear, we have
\begin{align*}
B^N_{s,t}\leq& C(t-s)^2\ll(1+\underset{v\leq T}{\sup}~|X^N_v|^2\rr),\\
F^N_{s,t}\leq& C(t-s)^2\ll|\frac1{\sqrt{N}}\sum_{j=1}^N U_j\rr|^2\ll(1+\underset{v\leq T}{\sup}~|X^N_v|^2\rr).
\end{align*}

Hence, for any $r\leq s\leq t\leq T,$
\begin{multline*}
(X^N_t-X^N_s)^2(X^N_s-X^N_r)^2\leq \ll(B^N_{s,t} + M^N_{s,t} + F^N_{s,t}\rr) \ll(B^N_{r,s} + M^N_{r,s}+ F^N_{r,s}\rr)\\
\leq C(t-r)^{4}\ll(1 +\ll|\frac1{\sqrt{N}}\sum_{j=1}^N U_j\rr|^{4}\rr)\ll(1 + \underset{v\leq T}{\sup}~|X^N_v|^{4}\rr)\\
+ C(t-r)^2\ll(1 +\ll|\frac1{\sqrt{N}}\sum_{j=1}^NU_j\rr|^2\rr)\ll(1 + \underset{v\leq T}{\sup}~|X^N_v|^2\rr)\ll(M^N_{s,t} + M^N_{r,s}\rr) + M^N_{s,t}M^N_{r,s}\\
=: \mathcal{R}^{N,r,t}_1 + \mathcal{R}^{N,r,s,t}_2 + \mathcal{R}^{N,r,s,t}_3.
\end{multline*}

Define $q_1,q_2>1$ such that $1/q_1 + 1/q_2 = 1$ and that $4 q_1<6$. Using H\"older's inequality,
\begin{align*}
\esp{\mathcal{R}^{N,r,t}_1}\leq& C(t-r)^{4}\ll(1+\esp{\ll|\frac1{\sqrt{N}}\sum_{j=1}^N U_j\rr|^{4 q_2}}^{1/q_2}\rr)\ll( 1 + \esp{\underset{v\leq T}{\sup}~|X^N_v|^{4 q_1}}^{1/q_1}\rr)\\
\leq& C(t-r)^{4},
\end{align*}
where we have used Lemmas~\ref{momentsomme}$.(ii)$ and~\ref{apriori}$.(ii)$ (with $p=3$ and $\kappa = 4 q_1$) to obtain the last line.

Besides, since for all $x,y\geq 0,$ $xy\leq x^2/2 + y^2/2,$
\begin{multline*}
\esp{\mathcal{R}^{N,r,s,t}_2}\\
\leq C(t-r)^2\ll\{\esp{\ll(1 +\ll|\frac1{\sqrt{N}}\sum_{j=1}^NU_j\rr|^{4}\rr)\ll(1 + \underset{v\leq T}{\sup}~|X^N_v|^{4}\rr)} + \esp{(M^N_{s,t})^2 + (M^N_{r,s})^2}\rr\}\\
\leq C(t-r)^2\ll(1  + \esp{\ll(M^N_{s,t}\rr)^2} + \esp{\ll(M^N_{r,s}\rr)^2}\rr),
\end{multline*}
where the last line has been obtained with the exact same reasoning as the one used for~$\mathcal{R}^{N,r,s,t}_1$ above. {Then by Burkholder-Davis-Gundy's inequality, for any $s\leq t\leq T,$
	\begin{align*}
	\espcc{\E}{(M^N_{s,t})^2}\leq& C\frac{1}{N^2}\espcc{\E}{\left(\sum_{j=1}^N U_j^2\int_{]s,t]\times\r_+}\uno{z\leq f(X^N_{u-})}d\pi^j(u,z)\right)^2}\\
	\leq& C\frac1N\sum_{j=1}^N U_j^4 \espcc{\E}{\left(\int_{]s,t]\times\r_+}\uno{z\leq f(X^N_{u-})}d\pi^j(u,z)\right)^2}\\
	\leq& C\frac1N\sum_{j=1}^N U_j^4 \left(\int_s^t \espcc{\E}{f(X^N_u)}du + \espcc{\E}{\ll(\int_s^t f(X^N_u)du\rr)^2}\right),
	\end{align*}
	where we have used the inequality~$(2.1.41)$ from Lemma~$2.1.7$ of \cite{jacod_discretization_2012} (the case $p=2$) to obtain the last line above. As a consequence, since $f$ has sublinear growth,
	$$\espcc{\E}{(M^N_{s,t})^2}\leq  C\frac1N\sum_{j=1}^N U_j^4 \left((t-s)\ll(1+\espcc{\E}{\underset{u\leq T}{\sup}|X^N_u|}\rr) + (t-s)^2\ll(1 + \espcc{\E}{\underset{u\leq T}{\sup}|X^N_u|^2}\rr)\right).$$
	Then, since $s\leq t\leq T,$ we can bound $(t-s)^2 \leq T(t-s),$ and obtain
	$$\espcc{\E}{(M^N_{s,t})^2}\leq C_T(t-s)\frac1N\sum_{j=1}^N U_j^4 \left(1 + \espcc{\E}{\underset{u\leq T}{\sup}|X^N_u|^2}\right).$$
	
	Then, by Cauchy-Schwarz' inequality, for all $s\leq t\leq T,$
	\begin{align}
	\esp{(M^N_{s,t})^2}\leq& C_T(t-s) \esp{\left(\frac1N\sum_{j=1}^N U_j^4\right)^2}^{1/2}\ll(1+\esp{\underset{u\leq T}{\sup}|X^N_u|^4}\rr)^{1/2}\nonumber\\
	\leq &C_T(t-s) \esp{\frac1N\sum_{j=1}^N U_j^8}^{1/2}\ll(1+\esp{\underset{u\leq T}{\sup}|X^N_u|^4}\rr)^{1/2}\nonumber\\
	\leq& C_T(t-s)\label{MNcarre},
	\end{align}
	by Lemma~\ref{apriori}$.(ii).$ In particular,
	$$\esp{\mathcal{R}^{N,r,s,t}_2}\leq C_T(t-r)^2(1 + (t-r)) \leq C_T(t-r)^2.$$
}

Now let us control $\mathcal{R}^{N,r,s,t}_3$ {using that
	$$\espcc{\E}{\mathcal{R}^{N,r,s,t}_3} = \espcc{\E}{M^N_{r,s}\espcc{\E}{M^N_{s,t}|\mathcal{F}_s}},$$
	with
	$$\mathcal{F}_s := \sigma(X^N_0)\vee \sigma(\pi^j_{|[0,s]\times\r_+}~:~1\leq j\leq N).$$
	By definition of Poisson measures, we know that $\pi^j_{|]s,t\times\r_+]}$ ($1\leq j\leq N$) are independent of~$\mathcal{F}_s.$ Thanks to this property, the following computation hold true. By Burkholder-Davis-Gundy's inequality,
	\begin{align*}
	\espcc{\E}{M^N_{s,t}|\mathcal{F}_s} :=& \frac1N\espcc{\E}{\ll.\ll|\sum_{j=1}^N U_j\int_{]s,t]\times\r_+}\uno{z\leq f(X^N_{u-})}d\tilde\pi^j(z,u)\rr|^2\rr|\mathcal{F}_s}\\
	\leq& C\frac1N\sum_{j=1}^N U_j^2\espcc{\E}{\ll.\int_{]s,t]\times\r_+}\uno{z\leq f(X^N_{u-})}d\pi^j(z,u)\rr|\mathcal{F}_s}\\
	=& C\frac1N\sum_{j=1}^N U_j^2 \int_s^t \espcc{\E}{f(X^N_u)|\mathcal{F}_s}du\\
	\leq& C(t-s)\frac1N\sum_{j=1}^N U_j^2\left(1+\espcc{\E}{\ll.\underset{u\leq T}{\sup}|X^N_u|\rr|\mathcal{F}_s}\right).
	\end{align*}
	
	Consequently,
	$$\espcc{\E}{\mathcal{R}^{N,r,s,t}_3}\leq C(t-s)\frac1N\sum_{j=1}^N U_j^2\espcc{\E}{M^N_{r,s} \left(1+\espcc{\E}{\ll.\underset{u\leq T}{\sup}|X^N_u|\rr|\mathcal{F}_s}\right)}.$$
	
	Then, applying Cauchy-Schwarz' inequality twice,
	$$\esp{\mathcal{R}^{N,r,s,t}_3}\leq C(t-s) \esp{(M^N_{r,s})^2}^{1/2} \esp{\left(\frac1N\sum_{j=1}^N U_j^2\right)^4}^{1/4} \ll(1+\esp{\underset{u\leq T}{\sup}|X^N_u|^4}\rr)^{1/4}.$$
	
	Using~\eqref{MNcarre},
	$$\esp{\mathcal{R}^{N,r,s,t}_3}\leq C(t-s)(s-r)^{1/2} \esp{\frac1N\sum_{j=1}^N U_j^8}^{1/4}\ll(1+\esp{\underset{u\leq T}{\sup}|X^N_u|^4}\rr)^{1/4} \leq C_T(t-r)^{3/2},$$
	thanks to Lemma~\ref{apriori}$.(ii).$
	
	Finally, we have
	$$\esp{(X^N_t - X^N_s)^2(X^N_s - X^N_r)^2} \leq \esp{\mathcal{R}^{N,r,t}_1 +\mathcal{R}^{N,r,s,t}_2+\mathcal{R}^{N,r,s,t}_3} \leq C_T(t-r)^{3/2},$$
	which proves the result.
}

\subsection{Proof of Theorem~\ref{mainresult}}\label{sectionproofmain}

We have just proved in the two previous sections that, for all $T>0,$
\begin{itemize}
	\item the finite-dimensional distributions of $(X^N_t)_{t}$ converge to those of~$(\bar X_t)_t$ on~$[0,T],$
	\item $\bar X_T - \bar X_{T-\delta}$ vanishes in distribution as $\delta$ goes to zero,
	\item for all $r\leq s\leq t\leq T, N\in\n^*,$
	$$\esp{\ll(X^N_s - X^N_r\rr)^2\ll(X^N_t - X^N_s\rr)^2}\leq C_T \ll(t-r\rr)^{3/2},$$
	for some positive constant~$C_T$.
\end{itemize}

Then, according to Theorem~13.5 of \cite{billingsley_convergence_1999}, we know that the process~$(X^N_t)_t$ converges in distribution to~$(\bar X_t)_t$ on $D([0,T],\r).$ As this holds true for any~$T>0,$ Lemma~16.3 of \cite{billingsley_convergence_1999} implies that the convergence in distribution holds true on the space~$D(\r_+,\r).$

{

\section{Proof of Proposition~\ref{quenched}}\label{proofquenched}

Proposition~\ref{quenched} is a consequence of the inequality~\eqref{quenchedfidi} (that was proved in Section~\ref{sectionCVfidi}) and of the law of the iterated logarithm. Indeed, recall that, for any $T>0,k\in\n^*$ and for all $0\leq t_1\leq...\leq t_k\leq T$ and $g_1,...,g_k$ belonging to~$C^3_b(\r)$,
\begin{multline*}
\ll|\espcc{\E}{g_1(X^N_{t_1})...g_k(X^N_{t_k})} - \espcc{\E}{g_1(\bar X^N_{t_1})...g_k(\bar X^N_{t_k})}\rr| \\
\leq ||g_1||_{3,\infty}...||g_k||_{3,\infty} C_{T,k} e^{C_{T,k}|W_{[N]}|}\ll(d_{KR}(\nu^N_0,\bar\nu_0) + \sum_{l=0}^{k-1}\espcc{\E}{\eps_N(T,X^N_{t_l})}\rr),
\end{multline*}
where~$\eps_N(t,x)$ is defined at~\eqref{diffsemigroup}. Let us point out that, to obtain a priori estimates of~$X^N$ in Lemma~\ref{apriori}, we have proved quenched a priori estimates at~\eqref{aprioriquenched}:
\begin{equation}\label{controlquenchbla}
\espcc{\E}{\ll(X^N_t\rr)^{2}}\leq C_{T}\ll(1 + \esp{(X^N_0)^{2}} + \frac1N\sum_{j=1}^N |U_j|^{4} + \ll|\frac1{\sqrt{N}}\sum_{j=1}^N U_j\rr|\rr)\exp\ll[TC\ll(1 + \ll|\frac1{\sqrt{N}}\sum_{j=1}^N U_j\rr|\rr)\rr].
\end{equation}
By respectively the strong law of large numbers and the law of the iterated logarithm we know that, almost surely,
$$\frac1N\sum_{j=1}^N |U_j|^{4} = \mathcal{O}(1)\textrm{ and }\ll|\frac1{\sqrt{N}}\sum_{j=1}^N U_j\rr| \leq 2\sqrt{\ln \ln N},$$
for $N$ large enough (i.e. larger than some random integer). In particular, almost surely,
$$\underset{t\leq T}{\sup}~\espcc{\E}{\ll(X^N_t\rr)^{2}} = \mathcal{O}\left([\ln N]^{C_T}\right).$$

Now by the definition of~$\eps_N$, almost surely
\begin{equation}\label{quenchedcontroleps}
\underset{t\leq T}{\sup}~\espcc{\E}{\eps_N(T,X^N_t)}\leq \mathcal{O}\ll([\ln N]^{C_T}\rr) e^{C_T|W^{[N]}|}\ll(\frac1{N\sqrt{N}}\sum_{j=1}^N|U_j|^3 + K\frac{\ln N}{\sqrt{N}} + \frac12\ll|\frac1N\sum_{j=1}^N U_j^2 - \sigma^2\rr|\rr).
\end{equation}

Then, recall that $W^{[N]} := \sigma\beta_N/\sqrt{N}$ (with $\beta$ some standard Brownian motion), and notice that by the law of the iterated logarithm (writting $\beta_N$ as the sum of the i.i.d. increasings of $\beta$), almost surely, $W^{[N]}\leq 2\sqrt{\ln \ln N}$ for $N$ large enough, and whence
$$e^{C_T|W^{[N]}|}=\mathcal{O}\ll([\ln N]^{C_T}\rr),$$
where the value of $C_T$ has changed. Finally, using again respectively the strong law of large numbers and the law of the iterated logarithm, we have, almost surely,
$$\frac1{N\sqrt{N}}\sum_{j=1}^N|U_j|^3 = \mathcal{O}\ll(N^{-1/2}\rr)\textrm{ and }\ll|\frac1N\sum_{j=1}^N U_j^2 - \sigma^2\rr| = \mathcal{O}\ll(\sqrt{\frac{\ln\ln N}{N}}\rr).$$

Using these bounds in~\eqref{quenchedcontroleps}, we obtain that, almost surely,
$$\underset{t\leq T}{\sup}~\espcc{\E}{\eps_N(T,X^N_t)} = \mathcal{O}\ll([\ln N]^{C_T}\cdot \frac{\ln N}{\sqrt{N}}\rr) =\mathcal{O}\ll([\ln N]^{C_T}\cdot N^{-1/2}\rr).$$

Using this last bound in~\eqref{controlquenchbla} proves Proposition~\ref{quenched}.

%
}

\begin{appendix}
\label{append}

\section{CLT coupling}\label{couplingproof}

Recall that the variables $U_j$ ($j\geq 1$) are i.i.d. centered with exponential moments: there exists~$\alpha>0,$
$$\esp{e^{\alpha|U_1|}}<\infty.$$

Then, according to Theorem~1 of \cite{komlos_approximation_1976}, it is possible to construct on the same probability space (possibly enlarged) a standard Brownian motion~$(\beta_t)_t$ for which there exist some positive constants~$\Gamma,\Lambda,\lambda$ such that, for any~$x>0,N\in\n^*,$
\begin{equation}\label{kmt}
\pro{\underset{k\leq N}{\max}\ll|\sum_{j=1}^k U_j - \sigma\beta_k\rr|> \Gamma \ln N + x}\leq \Lambda e^{-\lambda x},
\end{equation}
with $\sigma^2$ the variance of the variabes~$U_j$ ($j\geq 1$).

Then using the technics of Section~7.5 of \cite{ethier_markov_2005}, we have the following result (whose proof is written at the end of this subsection for self-containedness).

\begin{lem}\label{momentK}
	The random variable~$K$ defined as
	$$K := \underset{N\geq 2}{\sup}\ll|\sum_{j=1}^NU_j - \sigma\beta_N\rr|/\ln N$$
	admits exponential moments. More precisely, for any $0<\gamma<\lambda,$
	$$\esp{e^{\gamma K}}<\infty.$$
\end{lem}

Consequently, by definition of~$K$, we have that, almost surely, for all~$N\geq 2,$
$$\ll|\sum_{j=1}^NU_j - \sigma\beta_N\rr| \leq K\ln N.$$

Then, by defining the (non-independent) identically distributed Gaussian variables $W^{[N]} := \sigma\beta_N/\sqrt{N},$ we obtain exactly that, almost surely, for every~$N\geq 2,$
$$\ll|\frac{1}{\sqrt{N}}\sum_{j=1}^N U_j - W^{[N]}\rr|\leq K\frac{\ln N}{\sqrt{N}}.$$

{
\begin{rem}\label{noquench}
	Since $W^{[N]} := \sigma\beta_N/\sqrt{N}$ with $(\beta_t)_t$ some standard Brownian motion, it is known that, with probability equals one, the sequence $(W^{[N]})_N$ is not bounded, whence does not converge. This property prevents us to prove the quenched convergence with our approach.
\end{rem}
}

\begin{proof}[Proof of Lemma~\ref{momentK}]
	{\it Step~1.} In this first step, we prove the following control on the tails distribution of~$K:$ for all $x>0,$
	\begin{equation}\label{tail}
	\pro{K> 2(\Gamma +1) + x}\leq C_{\Lambda,\lambda} e^{-\lambda x}.
	\end{equation}
	
	This step is a mere rewriting of the proof of Corollary~7.5.5 of \cite{ethier_markov_2005}. To begin with, in~\eqref{kmt}, if we replace $x$ by $x+\ln N,$ we obtain
	\begin{equation*}
	\pro{\underset{k\leq N}{\max}\ll|\sum_{j=1}^k U_j - \sigma\beta_k\rr|> (\Gamma+1) \ln N + x}\leq \Lambda N^{-\lambda} e^{-\lambda x}.
	\end{equation*}
	
	Then
	\begin{align*}
	\pro{K> 2(\Gamma +1) + x}:=& \pro{\underset{N\geq 2}{\sup}\ll|\sum_{j=1}^NU_j - \sigma\beta_N\rr|/\ln N> 2(\Gamma +1) + x}\\
	\leq& \sum_{k=2}^{+\infty}\pro{\underset{2^{k-1}\leq N\leq 2^k}{\max}\ll|\sum_{j=1}^NU_j - \sigma\beta_N\rr|/\ln N> 2(\Gamma +1) + x}\\
	\leq & \sum_{k=2}^{+\infty}\pro{\underset{N\leq 2^k}{\max}\ll|\sum_{j=1}^NU_j - \sigma\beta_N\rr|/\ln (2^{k-1})> 2(\Gamma +1) + x}\\
	=& \sum_{k=2}^{+\infty}\pro{\underset{N\leq 2^k}{\max}\ll|\sum_{j=1}^NU_j - \sigma\beta_N\rr|> 2(\Gamma +1)\ln (2^{k-1}) + x\ln (2^{k-1})}\\
	\leq& \sum_{k=2}^{+\infty}\pro{\underset{N\leq 2^k}{\max}\ll|\sum_{j=1}^NU_j - \sigma\beta_N\rr|> (\Gamma +1)\ln (2^{k}) + x}\\
	\leq& \Lambda e^{-\lambda x}\sum_{k=2}^{+\infty} 2^{-k\lambda},
	\end{align*}
	which gives~\eqref{tail}.
	
	{\it Step~2.} Now, we conclude the proof using~\eqref{tail}. For $0<\gamma<\lambda,$
	\begin{align*}
	\esp{e^{\gamma K}}=& \int_0^{+\infty}\pro{e^{\gamma K}> t}dt = \int_0^{+\infty} \pro{K> \frac1\gamma \ln t}dt\\
	=&\int_0^{\exp[\gamma(2(\Gamma+1)+1)]} \pro{K> \frac1\gamma \ln t}dt + \int_{\exp[\gamma(2(\Gamma+1)+1)]}^{+\infty} \pro{K> \frac1\gamma \ln t}dt\\
	\leq& e^{\gamma(2(\Gamma+1)+1)} + \int_{\exp[\gamma(2(\Gamma+1)+1)]}^{+\infty} \pro{K> 2(\Gamma+1) + \ll(\frac1\gamma \ln t - 2(\Gamma+1)\rr)}dt.
	\end{align*}
	
	Then, we use~\eqref{tail} to bound the second term of the sum above
	$$\int_{\exp[\gamma(2(\Gamma+1)+1)]}^{+\infty} \pro{K> 2(\Gamma+1) + \ll(\frac1\gamma \ln t - 2(\Gamma+1)\rr)}dt
	\leq C_{\Lambda,\lambda} e^{2\lambda(\Gamma+1)}\int_{\exp[\gamma(2(\Gamma+1)+1)]}^{+\infty} t^{-\lambda/\gamma}dt,$$
	which is finite since~$\lambda/\gamma>1.$
\end{proof}

\section{Proofs of technical lemmas}
\label{prooftechnical}

\subsection{Proof of Lemma~\ref{momentsomme}}\label{sectionsomme}

The proof relies on the coupling between $S_N$ and $W^{[N]}$ defined in Appendix~\ref{couplingproof}.
$$\esp{e^{\gamma|S_N|}} \leq \esp{e^{\gamma|S_N - W^{[N]}|} e^{\gamma|W^{[N]}|}}\leq \frac12 \esp{e^{2\gamma K \ln N/\sqrt{N}}} + \frac12 \esp{e^{2\gamma|W^{[N]}|}},$$
where we have used that for all $x,y\geq 0,$ $xy\leq x^2/2 + y^2/2.$ Since $W^{[N]}$ is a Gaussian variable with parameters independent of~$N$, the second term of the sum above is finite and independent of~$N$. And, choosing $N_\gamma\in\n^*$ large enough such that $2\gamma(\ln N_\gamma)/\sqrt{N_\gamma}<\lambda$ (where $\lambda$ is defined in~\eqref{kmt}), we know that the second term is finite and bounded uniformly in $N\geq N_\gamma$ thanks to Lemma~\ref{momentK}. So the first statement of the lemma is proved.

And, since for all $x\geq 0,\beta>0$ and $p\in\n^*,$
$$x^p = \frac{p!}{\beta^p}\cdot\frac{\beta^px^p}{p!} \leq \frac{p!}{\beta^p}e^{\beta x},$$
the second statement of the lemma is a direct consequence of the first one.

\subsection{Proof of Lemma~\ref{apriori}}\label{sectionaprioriproof}

Let us prove Item~$(i).$ By Ito's formula,

\begin{align*}
\ll(X^N_t\rr)^{2p} =& (X^N_0)^{2p} + 2p\int_0^t (X^N_s)^{2p-1}b(X^N_s)ds \\
&+ \sum_{j=1}^N \int_{[0,t]\times\r_+}\ll[\ll(X^N_{s-}+ \frac{U_j}{\sqrt{N}}\rr)^{2p} - \ll(X^N_{s-}\rr)^{2p}\rr]\uno{z\leq f(X^N_{s-})}d\pi_j(s,z)\\
\espcc{\E}{\ll(X^N_t\rr)^{2p}}=&\esp{\ll(X^N_0\rr)^{2p}} + 2p\int_0^t\espcc{\E}{(X^N_s)^{2p-1}b(X^N_s)}ds\\
&+\sum_{j=1}^N\int_0^t
\espcc{\E}{f(X^N_s)\sum_{k=0}^{2p-1}\binom{2p}{k}(X^N_s)^k U_j^{2p-k} N^{-(2p-k)/2}}ds.
\end{align*}

Since
\begin{multline*}
\sum_{j=1}^N\sum_{k=0}^{2p-1}\binom{2p}{k}(X^N_s)^k \frac{U_j^{2p-k}}{N^{(2p-k)/2}}= \sum_{j=1}^N\sum_{k=0}^{2p-2}\binom{2p}{k}(X^N_s)^k \frac{U_j^{2p-k}}{N^{(2p-k)/2}} + \sum_{j=1}^N2p (X^N_s)^{2p-1}\frac{U_j}{\sqrt{N}}\\
\leq C_p\frac1N\sum_{j=1}^N \sum_{k=0}^{2p-2}|U_j|^{2p-k}|X^N_s|^{k} + C_p \frac{1}{\sqrt{N}}\sum_{j=1}^NU_j (X^N_s)^{2p-1}\\
\leq  C_p \ll(1+\frac1N\sum_{j=1}^N |U_j|^{2p}\rr)\ll(1 + |X^N_s|^{2p-2}\rr) + C_p \ll|\frac1{\sqrt{N}}\sum_{j=1}^N U_j\rr|\cdot \ll|X^N_s\rr|^{2p-1},
\end{multline*}

we have (recalling that $f$ and $b$ are sublinear),
\begin{align*}
\espcc{\E}{\ll(X^N_t\rr)^{2p}}\leq& \esp{(X^N_0)^{2p}} + C_pt + C_p\int_0^t\espcc{\E}{(X^N_s)^{2p}}ds\\
&+\ll(1+\frac1N\sum_{j=1}^N |U_j|^{2p}\rr)\ll( \int_0^t\ll(1+\espcc{\E}{\ll(X^N_s\rr)^{2p-1}}\rr)ds\rr)\\
&+ C_p\ll|\frac1{\sqrt{N}}\sum_{j=1}^N U_j\rr|\int_0^t\ll(1+\espcc{\E}{\ll(X^N_s\rr)^{2p}}\rr)ds.
\end{align*}

In order to control the term at the second line above, we use the fact that, for any $x,y\geq 0,$ $xy^{2p-1}\leq C_p(x^{2p}+y^{2p})$, whence the term at the second line is smaller than
\begin{multline*}
C_{T,p}\ll( 1 + \ll(\frac1N\sum_{j=1}^N |U_j|^{2p}\rr)^{2p}\rr) + C_{T,p}\int_0^t \espcc{\E}{(X^N_s)^{2p}}ds\\
\leq C_{T,p}\ll( 1 + \frac1N\sum_{j=1}^N |U_j|^{4p^2}\rr) + C_{T,p}\int_0^t \espcc{\E}{(X^N_s)^{2p}}ds.
\end{multline*}

Consequently
\begin{align*}
\espcc{\E}{\ll(X^N_t\rr)^{2p}}\leq &\esp{(X^N_0)^{2p}}+ C_{T,p}\ll(1 + \frac1N\sum_{j=1}^N |U_j|^{4p^2} + \ll|\frac1{\sqrt{N}}\sum_{j=1}^N U_j\rr|\rr) \\
&+ C_p\ll(1 + \ll|\frac1{\sqrt{N}}\sum_{j=1}^N U_j\rr|\rr)\int_0^t\espcc{\E}{\ll(X^N_s\rr)^{2p}}ds.
\end{align*}

By Gr\"onwall's lemma, we obtain, for all $t\leq T,$
\begin{multline}\label{aprioriquenched}
\espcc{\E}{\ll(X^N_t\rr)^{2p}}\leq C_{T,p}\ll(1 + \esp{(X^N_0)^{2p}} + \frac1N\sum_{j=1}^N |U_j|^{4p^2} + \ll|\frac1{\sqrt{N}}\sum_{j=1}^N U_j\rr|\rr) \\
\exp\ll[TC_p\ll(1 + \ll|\frac1{\sqrt{N}}\sum_{j=1}^N U_j\rr|\rr)\rr].
\end{multline}

Then, to prove Item~$(i),$ it is sufficient to integrate over the environment
$$\esp{\ll(X^N_t\rr)^{2p}}\leq C_{T,p}\ll(1 + \esp{|U_1|^{4p^2}} + \esp{\ll|\frac{1}{\sqrt{N}}\sum_{j=1}^N U_j\rr|}\rr)\esp{\exp\ll(C_{T,p}\ll|\frac{1}{\sqrt{N}}\sum_{j=1}^N U_j\rr|\rr)},$$
and to use Lemma~\ref{momentsomme}.

Now let us prove Item~$(ii).$ Let us rewrite the dynamics of $X^N$ in the following way
\begin{align*}
X^N_t =& X^N_0 + \int_0^t b(X^N_s)ds + \frac{1}{\sqrt{N}}\sum_{j=1}^N U_j\int_{[0,t]\times\r_+}\uno{z\leq f(X^N_{s-})}d\tilde \pi_j(s,z) \\
&+ \frac{1}{\sqrt{N}}\sum_{j=1}^N U_j\int_0^t f(X^N_s)ds,
\end{align*}
where $\tilde\pi_j(ds,dz) := \pi_j(ds,dz) -dsdz$ is the compensated Poisson measure of $\pi_j.$ Then,
\begin{align*}
\ll|X^N_t\rr|^{\kappa}\leq& C_\kappa |X^N_0|^{\kappa} + C_p t^{\kappa-1}\int_0^t |b(X^N_s)|^{\kappa} ds \\
&+ C_\kappa N^{-\kappa/2} \ll|\sum_{j=1}^N U_j\int_{[0,t]\times\r_+}\uno{z\leq f(X^N_{s-})}d\tilde \pi_j(s,z)\rr|^{\kappa} \\
&+ C_\kappa t^{\kappa-1}\ll|\frac{1}{\sqrt{N}}\sum_{j=1}^N U_j\rr|^{\kappa} \int_0^t f(X^N_s)^{\kappa}ds.
\end{align*}
So,
\begin{align}
\underset{t\leq T}{\sup}\ll|X^N_t\rr|^{\kappa}\leq& C_p |X^N_0|^{\kappa} + C_{T,p}\ll(1+ \int_0^T |X^N_s|^{\kappa} ds\rr) \nonumber\\
&+ C_\kappa N^{-\kappa/2}\underset{t\leq T}{\sup}~\ll|\sum_{j=1}^N U_j\int_{[0,t]\times\r_+}\uno{z\leq f(X^N_{s-})}d\tilde \pi_j(s,z)\rr|^{\kappa}\nonumber\\
&+ C_{T,\kappa}\ll|\frac{1}{\sqrt{N}}\sum_{j=1}^N U_j\rr|^{\kappa}\ll(1 + \int_0^T |X^N_s|^{\kappa} ds\rr)\nonumber
\end{align}

Whence
\begin{align}
\espcc{\E}{\underset{t\leq T}{\sup}~|X^N_t|^{\kappa}}\leq& C_{T,\kappa}\ll|\frac{1}{\sqrt{N}}\sum_{j=1}^N U_j\rr|^{\kappa} \ll( 1 + \int_0^T\espcc{\E}{|X^N_s|^{\kappa}}ds\rr)\label{supapriori}\\
&+ C_\kappa N^{-\kappa/2}\espcc{\E}{\underset{t\leq T}{\sup}~\ll|\sum_{j=1}^N U_j\int_{[0,t]\times\r_+}\uno{z\leq f(X^N_{s-})}d\tilde \pi_j(s,z)\rr|^{\kappa}}.\nonumber
\end{align}

{
To control the local martingale term above, we use, in the following order, Burkholder-Davis-Gundy's inequality, Jensen's inequality and inequality~$(2.1.41)$ of Lemma~$2.1.7$ of \cite{jacod_discretization_2012}:
	\begin{align*}
	&N^{-\kappa/2}\espcc{\E}{\underset{t\leq T}{\sup}~\ll|\sum_{j=1}^N U_j\int_{[0,t]\times\r_+}\uno{z\leq f(X^N_{s-})}d\tilde \pi_j(s,z)\rr|^{\kappa}}\\
	&\qquad\leq C_\kappa N^{-\kappa/2}\espcc{\E}{\ll(\sum_{j=1}^N U_j^2 \int_{[0,T]\times\r_+}\uno{z\leq f(X^N_{r-})}d\pi_j(r,z)\rr)^{\kappa/2}}\\
	&\qquad\leq C_\kappa N^{-1}\sum_{j=1}^N |U_j|^{\kappa} \espcc{\E}{\ll(\int_{[0,T]\times\r_+}\uno{z\leq f(X^N_{r-})}d\pi_j(r,z)\rr)^{\kappa/2}}\\
	&\qquad \leq C_\kappa N^{-1}\sum_{j=1}^N |U_j|^{\kappa}\ll(\espcc{\E}{\int_0^T f(X^N_r)dr} + \espcc{\E}{\ll(\int_0^T f(X^N_r)dr}\rr)^{\kappa/2}\rr)\\
	&\qquad \leq C_{T,\kappa}N^{-1}\sum_{j=1}^N |U_j|^{\kappa}\ll(1 + \int_0^T\espcc{\E}{f(X^N_r)^{\kappa/2}}dr\rr)\\
	&\qquad\leq C_{T,p}N^{-1}\sum_{j=1}^N |U_j|^{\kappa}\ll(1 + \int_0^T\espcc{\E}{f(X^N_r)^{p}}dr\rr),
	\end{align*}
	where the value of the constants $C_p$ and $C_{T,p}$ change from line to line above.
	
	Then, using that, for all $x,y\in\r_+,$ $xy\leq x^2/2 + y^2/2$ we have
	\begin{align*}
	&C_{T,p}N^{-1}\sum_{j=1}^N |U_j|^{\kappa}\ll(1 + \int_0^T\espcc{\E}{|X^N_r|^p}dr\rr)\\
	&\qquad\leq C_{T,p}\ll[\ll(N^{-1}\sum_{j=1}^N |U_j|^{\kappa}\rr)^2 + \ll(1 + \int_0^T\espcc{\E}{|X^N_r|^p}dr\rr)^2\rr]\\
	&\qquad \leq C_{T,p}\ll[N^{-1}\sum_{j=1}^N |U_j|^{2\kappa} + 1 + \int_0^T\espcc{\E}{|X^N_r|^{2p}}dr\rr].
	\end{align*}
	
	Then, we obtain from~\eqref{supapriori},
	\begin{align*}
	\esp{\underset{t\leq T}{\sup}|X^N_t|^{\kappa}}\leq&  C_{T,p}\esp{\ll|\frac{1}{\sqrt{N}}\sum_{j=1}^N |U_j|\rr|^{\kappa} \ll( 1 + \int_0^T\espcc{\E}{|X^N_s|^{\kappa}}ds\rr)}\\
	&+C_{T,p}\esp{N^{-1}\sum_{j=1}^N |U_j|^{2\kappa}} + 1 + \int_0^T\esp{|X^N_r|^{2p}}dr.
	\end{align*}
	
	Now, set $q_1,q_2>1$ such that $1/q_1 + 1/q_2 = 1$ and such that $\kappa q_1 < 2p$ (which is possible since $\kappa<2p$). By H\"older's inequality,
	\begin{align*}
	\esp{\underset{t\leq T}{\sup}|X^N_t|^{\kappa}}\leq&  C_{T,p}\esp{\ll|\frac{1}{\sqrt{N}}\sum_{j=1}^N |U_j|\rr|^{\kappa q_2}}^{1/q_2}\esp{\ll( 1 + \int_0^T\espcc{\E}{|X^N_s|^{\kappa}}ds\rr)q_1}^{1/q_1}\\
	&+C_{T,p}\esp{N^{-1}\sum_{j=1}^N |U_j|^{2\kappa}} + 1 + \int_0^T\esp{|X^N_r|^{2p}}dr\\
	\leq& C_{T,p}\esp{\ll|\frac{1}{\sqrt{N}}\sum_{j=1}^N |U_j|\rr|^{\kappa q_2}}^{1/q_2}\esp{\ll( 1 + \int_0^T\espcc{\E}{|X^N_s|^{\kappa q_1}}ds\rr)}^{1/q_1}\\
	&+C_{T,p}\esp{N^{-1}\sum_{j=1}^N |U_j|^{2\kappa}} + 1 + \int_0^T\esp{|X^N_r|^{2p}}dr.
	\end{align*}
}

Recalling that $\kappa q_1<2p$, the result of Item~$(ii)$ is then a consequence of Lemma~\ref{momentsomme} and Item~$(i)$.

\subsection{Proof of Lemma~\ref{controlsemigroup}}\label{proofcontrolsemi}

Let us recall that, in all this proof, we work conditionally on the environment variable~$W$. This proof is very similar to that of Proposition~3.4 of \cite{erny_mean_2022}. We still write the proof here to have the dependency of the control w.r.t.~$W$.

Under Assumption~\ref{hypfun}, the assumptions of Theorem~1.4.1 of \cite{kunita_lectures_1986} are satisfied for the flow of the limit process~$(\bar X^{(x)}_t)_t$. Indeed, the local characteristics $(\bar b,\bar a)$ is given by
$$\bar b(x) = b(x) + Wf(x)\textrm{ and }\bar a(x,y)=\sigma^2\sqrt{f(x)f(y)}.$$

Assumptions~1 and~2 of \cite[p. 8, 9]{kunita_lectures_1986} are obtained directly from a priori estimates on the second order moments of the process (recalling that we work conditionally on the environment, this point is very classical and therefore omitted). And Assumption~3 (p.~15) is guaranteed by the Lipschitz continuity of~$b,f,\sqrt{f}$ (which is weaker than Assumption~\ref{hypfun}). And the other assumption of Theorem~1.4.1 is exactly Assumption~\ref{hypfun}.

As a consequence, we know, by Theorem~1.4.1 of \cite{kunita_lectures_1986}, that for all $t\geq 0,$ the function $x\mapsto \bar X_t^{(x)}$ is~$C^3.$ Then, to be able to deduce the regularity of the function
$$x\mapsto \bar P_{\E,t}g(x) = \espcc{\E}{\bar X^{(x)}_t},$$
we just have to differentiate under the expectation. To prove the domination condition for the equations below to be true, we prove a priori estimates on the derivatives of the stochastic flow.
\begin{align}
\ll(\bar P_{\E,t}g\rr)'(x) = & \espcc{x}{\partial_x\bar X^{(x)}_tg'(\bar X^{(x)}_t)}\nonumber\\
\ll(\bar P_{\E,t}g\rr)''(x) = &\espcc{x}{\partial^2_x\bar X^{(x)}_tg'(\bar X^{(x)}_t) + \ll(\partial_x\bar X^{(x)}_t\rr)^2g''(\bar X^{(x)}_t)}\nonumber\\
\ll(\bar P_{\E,t}g\rr)'''(x) = &\espcc{x}{\partial^3_x\bar X^{(x)}_tg'(\bar X^{(x)}_t) + 3\partial^2_x\bar X^{(x)}_s\partial_x\bar X^{(x)}_tg''(\bar X^{(x)}_t) + \ll(\partial_x\bar X^{(x)}_t\rr)^3g'''(\bar X^{(x)}_t)}\label{semigroup3}
\end{align}

\begin{rem}
	To differentiate under the expectation, instead of proving classical domination conditions like
	$$\esp{\underset{x}{\sup}~\ll|\partial_x\bar X^{(x)}_t\rr|}<\infty,$$
	we rather verify the following uniformly integrable sufficient kind of conditions
	$$\underset{x}{\sup}~\esp{\ll|\partial_x\bar X^{(x)}_t\rr|^2}<\infty.$$
	See Lemma~3.1 of \cite{erny_mean_2022-1} (inspired from Lemma~6.1 of \cite{eldredge_strong_2018}) for a formal statement of this result.
\end{rem}

To begin with, we know that the first derivative of the stochastic flow is solution to the following SDE.

$$\partial_x \bar X_t^{(x)} = 1 + \int_0^t\partial_x \bar X^{(x)}_s\ll(b'(\bar X_s^{(x)}) + Wf'(\bar X_s^{(x)})\rr)ds + \sigma\int_0^t\partial_x \bar X^{(x)}_s(\sqrt{f})'(\bar X^{(x)}_s)dB_s.$$

For any $p\in\n^*,$ by Ito's formula,
\begin{align*}
\espcc{\E}{\ll(\partial_x \bar X_t^{(x)}\rr)^{2p}}=& 1 + \int_0^t\mathbb{E}_\E(\partial_x\bar X^{(x)}_s)^{2p}\ll(2pb'(\bar X^{(x)}_s) + 2pWf'(\bar X^{(x)}_s) + p(2p-1)\sigma^2(\sqrt{f})'(\bar X^{(x)}_s)\rr)ds\\
\leq& 1 + C_p(1+|W|)\int_0^t\espcc{\E}{\ll(\partial_x \bar X_t^{(x)}\rr)^{2p}}ds.
\end{align*}

By Gr\"onwall's lemma,
\begin{equation}\label{fstder}
\espcc{\E}{\ll(\partial_x \bar X_t^{(x)}\rr)^{2p}} \leq e^{C_{p,t}(1+|W|)}.
\end{equation}

Let us prove a similar result on the second derivative of the stochastic flow.

\begin{align*}
\partial^2_x \bar X^{(x)}_t = & \int_0^t \partial^2_x\bar X^{(x)}_s\ll(b'(\bar X^{(x)}_s) + Wf'(\bar X^{(x)}_s)\rr)ds + \int_0^t \ll(\partial_x\bar X^{(x)}_s\rr)^2\ll(b''(\bar X^{(x)}_s) + Wf''(\bar X^{(x)}_s)\rr)ds\\
&+\sigma\int_0^t \partial^2_x \bar X^{(x)}_s (\sqrt{f})'(\bar X^{(x)}_s)dB_s + \sigma\int_0^t \ll(\partial_x \bar X^{(x)}_s\rr)^2 (\sqrt{f})''(\bar X^{(x)}_s)dB_s.
\end{align*}

Then, by Ito's formula,
\begin{align*}
\espcc{\E}{\ll(\partial^2_x \bar X^{(x)}_t\rr)^{2p}}=& 2p\int_0^t \mathbb{E}_\E \ll(\partial^2_x \bar X^{(x)}_s\rr)^{2p}\ll(b'(\bar X^{(x)}_s) + Wf'(\bar X^{(x)}_s)\rr)ds\\
&+ 2p\int_0^t \mathbb{E}_\E \ll(\partial^2_x \bar X^{(x)}_s\rr)^{2p-1}\ll(\partial_x \bar X^{(x)}_s\rr)^2\ll(b''(\bar X^{(x)}_s) + Wf''(\bar X^{(x)}_s)\rr)ds\\
&+p(2p-1)\sigma^2\int_0^t \mathbb{E}_\E \ll(\partial^2_x \bar X^{(x)}_s\rr)^{2p}\ll((\sqrt{f})'(\bar X^{(x)}_s)\rr)^2ds\\
&+p(2p-1)\sigma^2\int_0^t \mathbb{E}_\E \ll(\partial^2_x \bar X^{(x)}_s\rr)^{2p-2}\ll(\partial_x \bar X^{(x)}_s\rr)^4\ll((\sqrt{f})''(\bar X^{(x)}_s)\rr)^2ds.
\end{align*}

Then using that for all $x,y,z\geq 0,$
$$x^{2p-1}y \leq C_p (x^{2p} + y^{2p})\textrm{ and }x^{2p-2}z\leq C_p(x^{2p} + z^{2p}),$$
we obtain from the previous inequality
\begin{align*}
\espcc{\E}{\ll(\partial^2_x \bar X^{(x)}_t\rr)^{2p}}\leq& C_p(1+|W|)\int_0^t\mathbb{E}_\E \ll(\partial^2_x \bar X^{(x)}_s\rr)^{2p}ds + C_p\ll(1 + |W|\rr)\int_0^t\mathbb{E}_\E \ll(\partial_x \bar X^{(x)}_s\rr)^{4p}ds\\
&+ C_p \int_0^t \mathbb{E}_\E \ll(\partial^2_x \bar X^{(x)}_s\rr)^{2p}ds + C_p\int_0^t\mathbb{E}_\E \ll(\partial_x \bar X^{(x)}_s\rr)^{8p}ds.
\end{align*}

Thanks to~\eqref{fstder},
$$\espcc{\E}{\ll(\partial^2_x \bar X^{(x)}_t\rr)^{2p}}\leq C_{p,t}e^{C_{p,t}|W|} + C_p\ll(1+|W|\rr)\int_0^t \espcc{\E}{\ll(\partial^2_x \bar X^{(x)}_s\rr)^{2p}}ds,$$
whence, by Gr\"onwall's lemma,
\begin{equation}\label{sndder}
\espcc{\E}{\ll(\partial^2_x \bar X^{(x)}_t\rr)^{2p}}\leq C_{p,t}e^{C_{p,t}|W|}.
\end{equation}

With the same reasonning, we can also prove that
\begin{equation}\label{thrder}
\espcc{\E}{\ll(\partial^3_x \bar X_t^{(x)}\rr)^{2p}} \leq C_{p,t} e^{C_{p,t}|W|}.
\end{equation}

Finally, using~\eqref{fstder},~\eqref{sndder} and~\eqref{thrder} in the formulae~\eqref{semigroup3} allows to prove the result.

\end{appendix}

\medskip

{\bf Acknowledgments.} The author would like to thank an anonymous referee for suggesting to use the law of the iterated logarithm to study the quenched properties of the model.

\bibliography{biblio}       


\end{document}